\documentclass[12pt,reqno,a4paper]{amsart}
\usepackage{xcolor}
\usepackage{amsmath}
\usepackage{amsthm}
\usepackage{amssymb}
\usepackage{mathrsfs}
\usepackage{graphicx}
\usepackage{tikz}
\usepackage{bm}
\usepackage{ucs}
\usepackage{wrapfig}
\usepackage{color}
\usepackage{float}
\usepackage{fullpage}
\usepackage{epstopdf}
\usepackage[breaklinks,colorlinks=true,linkcolor=blue,citecolor=red,
  backref=page]{hyperref}

\newtheorem{remark}{Remark}

\newtheorem{lemma}{Lemma}
\newtheorem{theorem}{Theorem}
\newtheorem{corollary}{Corollary}
\newtheorem{definition}{Definition}

\newcommand{\R}{\mathbb{R}}

\newcommand{\ind}{\,\mbox{d}}
\newcommand{\normmm}[1]{{\left\vert\kern-0.3ex\left\vert\kern-0.25ex\left\vert #1 
    \right\vert\kern-0.25ex\right\vert\kern-0.3ex\right\vert}}

\begin{document}
\date{}

\title{Single Mode Multi-frequency Factorization Method for the Inverse Source Problem  in Acoustic Waveguides}
%\title{Multi-frequency Factorization Method for the Inverse Source Problem  in Single Mode Waveguide}
%\author{Shixu Meng}
\author[]{Shixu Meng$^1$}\footnotetext[1]{Academy of Mathematics and Systems Science, Chinese Academy of Sciences,
Beijing 100190, China.  {\tt shixumeng@amss.ac.cn}}
%\author[]{Bo Zhang $^2$}
%\footnotetext[2]{Academy of Mathematics and Systems Science, Chinese Academy of Sciences,
%Beijing 100190, China and School of Mathematical Sciences, University of Chinese Academy of Sciences,
%Beijing 100049, China.  {\tt b.zhang@amt.ac.cn}}
\begin{abstract}
This paper investigates the inverse source problem with a single propagating mode at multiple frequencies in an acoustic waveguide. The goal is to provide both theoretical justifications and efficient algorithms for imaging extended sources using the sampling methods. In contrast to the existing far/near field operator based on the integral over the space variable in the sampling methods,  a multi-frequency far-field operator  is introduced based on the integral over the frequency variable. This far-field operator is defined in a way to incorporate the possibly non-linear dispersion relation, a unique feature in waveguides. The factorization method is deployed to establish a rigorous characterization of the {range support which is the support of source in the direction of wave propagation}. A related factorization-based sampling method is also discussed. These sampling methods are shown to be capable of imaging the range support of the source. Numerical examples are provided to illustrate the performance of the sampling methods, including an example to image a complete sound-soft block.
\end{abstract} 

\maketitle

\section{Introduction}  \label{section intro}

Inverse scattering is of great importance in non-destructive testing, medical imaging, geophysical exploration and numerous problems associated with target identification. In the last thirty years, sampling methods such as the linear sampling method  \cite{ColtonKirsch}, the factorization method  \cite{Kirsch98} and their extensions  have attracted a lot of interests. These sampling methods do not require a priori information about the scattering objects, and provide both theoretical justifications and robust numerical algorithms. 
%Sampling methods, such as linear sampling method and factorization method \cite{Kirsch98,kirsch2008factorization,CK,CaCo} do not require a priori information about the scattering objects, and provide both theoretical justifications and robust numerical algorithms. There have been recent interests in orthogonal sampling \cite{potthast2010study,griesmaier2011multi,harris2020orthogonality}, direct sampling method \cite{LiuIP17}, reverse time migration \cite{CCH2013} and other type of sampling methods.
%Classical methods include the linear sampling method  \cite{ColtonKirsch}
%and the factorization method  \cite{Kirsch98}.
 There have been recent interests in inverse scattering for waveguides, mainly motivated by their applications in ocean acoustics,  non-destructive testing of slender structures, imaging in and of tunnels \cite{baggeroer1993overview,haack1995state,rizzo2010ultrasonic}. One of the  early works is the generalized dual space indicator method for underwater imaging
\cite{xu2000generalized}. The linear sampling method and factorization method were studied in acoustic waveguides \cite{arens2011direct,bourgeois2008linear,bourgeois2012use,monk2012sampling} and in elastic waveguides \cite{bourgeois2011use,bourgeois2013use}.  Theories and applications of sampling methods have been extended to periodic waveguides \cite{bourgeois2014identification,sun2013reconstruction}, a time domain linear sampling method \cite{monk2016inverse}, sampling methods in electromagnetic waveguide \cite{meng2021,monk2019near}. It is also worth mentioning related imaging methods such as the time migration imaging method in an acoustic terminating waveguide \cite{tsogka2017imaging},  in an electromagnetic waveguide \cite{chen2017direct} and the time reversal imaging in an electromagnetic terminating waveguide \cite{borcea2015imaging}. Relations between the time migration imaging and the factorization method as well as a factorization-based imaging method is discussed in \cite{borcea2019factorization}. We also refer to \cite{bellis13} for related discussions on the connections among different imaging methods in the homogeneous space case. It is worth mentioning that a scatterer may be invisible in a waveguide that supports one propagating mode at a single frequency \cite{dhia2017perfect}. 

One difference between the waveguide and the homogeneous whole space is that the waveguide supports finitely many propagating modes and infinitely many evanescent modes at a fixed frequency. The theories of the linear sampling method and factorization method in waveguides \cite{arens2011direct,BORCEA2019556,borcea2019factorization,bourgeois2008linear,bourgeois2012use,monk2012sampling} have been established using both propagating modes and evanescent modes. Particularly in the far-field, the measurements stem from the propagating modes so that the linear sampling method and factorization method are usually implemented using only the propagating modes. The imaging results are good as far as there are sufficiently many  propagating modes. However, an interesting observation made in \cite{monk2012sampling} is that the linear sampling method may not be capable of imaging  a complete block ({ where a complete block could be an obstacle spanning the entire cross-section}) using  far-field measurements at a single frequency.  This seems not a problem due to the sufficiency or deficiency of propagating modes. This motivates us to investigate further the role of propagating modes. One fact about the propagating modes is that each propagating mode propagates with different group velocity, indicated by the dispersion relation which is a unique feature in waveguide. 
%At different frequencies, different propagating modes (at a fixed branch) indeed have intrinsic properties that may lead to the understanding of imaging with a single propagating mode at multi-frequencies. 
Intuitively, a single propagating mode at multi-frequencies may help understand further the problem \cite{monk2012sampling} of imaging a complete block with far-field measurements, since the phase (travel time) provides information on the bulk location.  The efficiency of using  a single propagating mode may be of practical interests in long-distance communication optical devices or tunnels.  The above discussions lead to the following question in waveguide imaging: how to make use of a single propagating mode at multiple frequencies in a mathematically rigorous way? 
%To answer this, we first briefly survey sampling methods using multiple frequency measurements in waveguide.

While most of the existing works focus on the sampling methods at a single frequency with full-aperture measurements, the time migration imaging method \cite{tsogka2017imaging} and the linear sampling method \cite{BORCEA2019556} investigated the limited-aperture problem with multiple frequencies. The time domain linear sampling method \cite{monk2016inverse} amounts to use multiple frequency measurements. However, there seems no work to treat the problem with less measurements (such as backscattering measurements obtained from a single propagating mode) at multiple frequencies to lay the  theoretical foundation for extended scatterers. This is one of the motivations of this paper: to provide theoretical foundation for extended sources using the sampling methods with a single propagating mode at multiple frequencies. This necessitates the study of the factorization method which provides theoretical justifications for extended sources.   Another motivation  is to help understand further the problem of imaging  a complete block of the waveguide as pointed out in \cite{monk2012sampling}. Though the single frequency linear sampling method may not be capable of imaging  a complete block using  far-field measurements, measurements with a single mode at multiple frequencies may help locate such a block efficiently. { In a certain frequency range, the waveguide supports a single propagating mode so that the far-field measurements stem from such single propagating mode; such a singe mode waveguide was also studied \cite{chesnel2018} for applications such as invisibility and perfect reflectivity.}  Based on these two motivations, this paper considers the factorization method {for the inverse source problem} with   \textit{single mode multi-frequency measurements} for a waveguide that supports a single propagating mode; {the measurements are made at one opening of the waveguide, and we call these ``backscattering'' measurements (where ``'' is to distinguish the inverse source problem from the inverse scattering problem).} A related factorization-based sampling method would also be discussed. {These sampling methods are shown to be capable of imaging the range support of the source, but not the support in the cross-section direction. With a single propagating mode, it seems to us that only the support in one direction could be imaged. It was similarly reported in \cite{GriesmaierSchmiedecke-source} that the support in one direction can be imaged with far-field measurements at a single direction (and its opposite direction).}
%The result with a single propagating mode would be fundamental to investigate multiple  propagating modes to establish theoretical justifications for imaging extended scatterers in  inverse obstacle/medium problems.

%Furthermore, it seems that using sampling methods to image of a complete block of the waveguide is an interesting problem \cite{monk2012sampling}. This brings the second motivation of this paper: to efficiently locate a possibly complete block of the waveguide. It seems intuitively that the  multi-frequency measurements with a single mode should be capable of locating a block, since the phase (travel time) provides information on the bulk location.  The efficiency to use  a single propagating mode is of practical interests in long-distance communication optical devices or tunnels. 

The methods in this paper inherit the tradition of the classical sampling methods to define a data operator in the far-field. However, in contrast to the existing far/near field operator based on the integral over the space variable,  a \textit{multi-frequency far-field operator}  is introduced based on the integral over the frequency variable.  This operator has a symmetric factorization which allows us to investigate the factorization method and the factorization-based sampling method. Sampling methods to treat the inverse source problems in the whole space were studied in \cite{AlaHuLiuSun,GriesmaierSchmiedecke-source,LiuMeng2021}, including the factorization method and direct/orthogonal sampling method. One difference between the waveguide imaging in this paper and the imaging in the whole space arises from the dispersion relation, a unique feature in waveguide. This brings an additional difficulty in defining or factorizing the multi-frequency far-field operator. This paper overcomes this difficulty by incorporating the dispersion relation into the definition of the far-field operator.  Another difference is that the measurements under consideration are ``backscattering'' measurements (in contrast to measurements obtained in two opposite directions in the homogeneous space case \cite{GriesmaierSchmiedecke-source}); we shall show how to use these less measurements to design appropriate multi-frequency operators to deploy the factorization method. 
As demonstrated in \cite{AlaHuLiuSun,GriesmaierSchmiedecke-source,LiuMeng2021} and the references therein, one may hope to image an approximate support of the source with only one measurement point/direction for imaging in the whole space. This paper reveals a similar result in the waveguide imaging problem: it turns out that the range support of the source can be reconstructed with a single propagating mode at multiple frequencies.

This paper is further organized as follows. Section \ref{section model} provides the mathematical model  for the inverse source problem in a waveguide. We also briefly discuss the concept of propagating/evanescent modes and dispersion relation, as well as the Green function and the forward problem. The multi-frequency far-field operator is proposed in Section \ref{section operator}, and it is shown that it has a symmetric factorization where the factorized middle operator is coercive. Section \ref{section factorization method} is devoted to the factorization method which provides a mathematically rigorous characterization of the source support. Motivated by the factorization of the multi-frequency far-field operator, a factorization-based sampling method is then discussed in Section \ref{section factorization based sampling method}. Finally numerical examples are provided in Section \ref{section numerical example} to illustrate the performance of the two sampling methods to image extended sources. As an extension and application, a numerical example is also provided to illustrate the potential of the multi-frequency sampling methods to image a complete sound-soft block. { Section \ref{section design multi-frequency operator and source assumption} discusses another two multi-frequency operators for sampling methods which impose less restrictive condition on the source. Finally we conclude this paper with a conclusion in Section \ref{section conclusion}.}

%%%%%%
\section{Mathematical model and problem setup} \label{section model}

Let us now introduce the mathematical model of the inverse source problem. The   waveguide in $\R^d$ is given by $W=(-\infty, \infty) \times \Sigma$, where $\Sigma$ is the cross-section of the waveguide. In this paper we restrict ourself to the two dimensional case $d=2$ {where $\Sigma=(0,|\Sigma|)$ with $|\Sigma|$ denoting the length of $\Sigma$}, and remark that the extension to the three dimensional case can follow similarly. We consider the model that the boundary of the waveguide is either sound-hard, sound-soft, or mixed (Dirichlet on part of the boundary and Neumann on the remaining boundary).  Let us denote the support of the source $f$ by $D$ with the assumption that $f\in L^\infty(D)$ and $D$ is a domain with Lipschitz boundary. See Figure \ref{Forward perturbation} for an illustration of the problem. For any $x \in \R^d$, let $x=(x_1,x_\perp)$ be the coordinate of $x$ where $x_1$ is the variable in the range and $x_\perp$ is the variable in the cross-section. 

\begin{figure}[hb!] 
        \centering
            \begin{tikzpicture}[scale=0.3]
\draw (-15,5)  -- (10,5);
\draw (-15,-1)  -- (10,-1); 
%\draw (15,-1)  -- (15,5); 
\draw[fill=gray, opacity=0.7]     plot [smooth cycle] coordinates          { (-1,1) ( 0,3) (1, 2) (0,1)};
\draw  (-10,1) node[above, black]{\tiny $W=(-\infty, \infty) \times \Sigma$ };
\draw  (0.5,1.3) node[right, black]{\tiny $D$ };
            \end{tikzpicture} 

     \caption{
     \linespread{1}
     Source $D$ in the acoustic waveguide $W$.\label{Forward perturbation}
     } 
\end{figure}
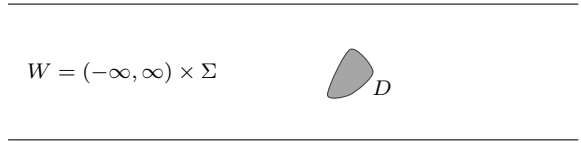

Let the interrogating wave number be a positive number $k$.  Let  $u^s(\cdot;k)$ be the wave field due to the source $f$. The wave field $u^s(\cdot;k)$ belongs to $\widetilde{H}^1_{loc}(W)$ and satisfies 
\begin{eqnarray}
\Delta u^s(\cdot;k) + k^2 u^s(\cdot;k) = { -}f \quad &&\mbox{in} \quad W, \label{ForwardUnbounded1} \\
\mathcal{B} u^s(\cdot;k) = 0 \quad &&\mbox{on} \quad \partial W, \label{ForwardUnbounded2} \\
u^s(\cdot;k) \quad && \mbox{satisfies a radiation condition}, \label{ForwardUnbounded3}
\end{eqnarray}
where the radiation condition is specified in Section \ref{subsection mode decomposition}, and $\mathcal{B}$ denotes either the Dirichlet, Neumann, or a mixed boundary condition, i.e.
\begin{equation} \label{def boundary condition}
\mathcal{B} u^s:=
\Bigg\{
\begin{array}{ccc}
u^s=0  &    \mbox{Dirichlet}   \\
\frac{\partial u^s}{\partial \nu}=0  &    \mbox{Neumann}  \\
\mbox{Dirichlet on } \Gamma_1,  \mbox{ Neumann on } \Gamma_2 &     \mbox{mixed} 
\end{array},
\end{equation}
with {$\Gamma_1$ (resp. $\Gamma_2$) denoting either the top (resp. bottom)  or the bottom (resp. top) boundary}, $\nu$ is the unit outward normal,
and $\widetilde{H}^1_{loc}(W)$ is the set 
$$\{u: u|_{\Omega} \in H^1(\Omega), \forall \,  {   \Omega=(-b,b) \times \Sigma}   \mbox{ s.t. } D \subset \Omega\}
$$ with $H^1$ denoting the standard Sobolev space consisting of bounded $L^2$ functions whose derivatives belong to $L^2$.

 The goal is to reconstruct the support $D$ of the source  with a single propagating mode at multiple frequencies. % We initiate the study by considering a waveguide that supports only one propagating mode, which may be of practical interests in long-distance communication optical devices or tunnels. The use of one propagating mode may also lead to efficiency in detecting a complete block of the waveguide. 
In the case that the waveguide supports one propagating mode, the use of one propagating mode at multiple frequencies amounts to the use of multi-frequency measurements  at one measurement point. In particular, the inverse problem is to characterize the support $D$ of the source  from the following { non-vanishing} measurements

\begin{itemize}
\item ``Backscattering'' measurements:
\begin{equation} \label{measurements}
\{u^s_p({x^*};k): \quad {x^*} \in \{-a\} \times \Sigma,\, k \in K\}
\end{equation}
where $u^s_p$ represents the propagating part (far-field) of $u^s$,  ${x^*}$ is a measurement point at a measurement cross section $\{-a\} \times \Sigma$, and $K$ is the range of wavenumber (i.e. frequency).
\end{itemize}
Here $K$ is the range of wavenumber such that the waveguide supports only one propagating mode (which will be specified later in details in Section \ref{subsection mode decomposition}). In practice, $a>0$ is such that the measurements are at the far-field of the waveguide {(where the contribution from the evanescent part is negligible)}, this is feasible as the frequencies are always discrete numbers and the evanescent part will be negligible for some $a$. In this paper we fix the notation that  waves measured at $\{-a\} \times \Sigma$ are called ``backscattering''  measurements (to distinguish the inverse source problem from the inverse scattering problem)  and waves measured at $\{a\} \times \Sigma$ are called ``forward scattering'' measurements.  It is also possible to consider  both the ``backward and forward scattering'' measurements, i.e. measurements at $\{\pm a\} \times \Sigma$, however this type of measurements may not be applicable in applications where   measurements can only be made at one opening of the waveguide. Nevertheless, a detailed remark  is to be made in Section \ref{section design multi-frequency operator and source assumption} concerning this type of measurements. 
%The design of another (less obvious) multi-frequency operator which requires less restrictive assumptions on the source is also to be discussed in \ref{section design multi-frequency operator and source assumption}.

%The fact that only one measurement point is available poses difficulty in the analysis and design of sampling methods. This is the extreme case of limited-aperture problems where data is available on part of the measurement cross section. However, with a certain range of frequency, one may hope to propose  a sampling method with only single mode measurement. Here we remark that the source may be invisible in a waveguide that supports one propagating mode at a single frequency \cite{dhia2017perfect}.

%\section{Propagating modes, dispersion relation, and Green function} \label{section waveguide prop}
\subsection{Propagating modes} \label{subsection mode decomposition}
Let us introduce the eigensystem $\{\psi_n,\lambda_n^2\}_{n=1}^\infty$ (where we sort $0\le\lambda_n \le \lambda_{n+1}$) by
\begin{eqnarray}
\Delta_{\perp} \psi_n + \lambda_n^2 \psi_n = 0 \quad &&\mbox{in} \quad \Sigma, \\
\mathcal{B}_\perp \psi_n =0 \quad &&\mbox{on} \quad \partial \Sigma,
\end{eqnarray}
where
$\mathcal{B}_\perp$ (defined according to $\mathcal{B}$ in \eqref{def boundary condition}) denotes either the Dirichlet, Neumann, or mixed boundary condition, i.e.
\begin{equation*}
\mathcal{B}_\perp \psi_n=
\Bigg\{
\begin{array}{ccc}
\psi_n=0  &    \mbox{Dirichlet}   \\
\frac{\partial \psi_n}{\partial \nu_\perp}=0  &    \mbox{Neumann}  \\
\mbox{Dirichlet on } \Gamma_{1\Sigma},  \mbox{ Neumann on } \Gamma_{2\Sigma} &     \mbox{mixed} 
\end{array},
\end{equation*}
where $\nu_\perp$ denotes the unit outward normal to $\partial \Sigma$. {Here $\Gamma_{1\Sigma}$ is either the top $\{x_\perp: x_\perp = |\Sigma|\}$ or the bottom $\{x_\perp: x_\perp = 0\}$ part of $\partial \Sigma$ according to \eqref{def boundary condition}, and $\Gamma_{2\Sigma}$ is the remaining part of $\partial \Sigma$.}

For every fixed $k$ with $k^2 \not \in \{\lambda_n^2\}_{n=1}^\infty$,  {   the following physical wave functions (can be found via separation of variables) are solutions to the Helmholtz equation \eqref{ForwardUnbounded1} away from the support of the source}
\begin{equation}
%e^{i\sqrt{k^2-\lambda_n^2} |x_1|} \psi_n(x_\perp)=
\Bigg\{
\begin{array}{ccc}
e^{ \pm i \sqrt{k^2-\lambda_n^2} x_1} \psi_n(x_\perp),  &    k>\lambda_n   \\
e^{-\sqrt{\lambda_n^2-k^2} |x_1|} \psi_n(x_\perp),  &    k<\lambda_n  
\end{array}
, \quad n=1,2,\cdots, \quad {|x_1| \gg 1},
\end{equation}
here we have chosen the branch cut of the square root with positive imaginary part, and we call it a  propagating mode if $k>\lambda_n$ and an evanescent mode if $k<\lambda_n$. 

{
%We note that the evanescent modes are exponentially decaying so that what we measure in the far-field stems from  the propagating modes. 
For our problem, we consider the  frequency  range $(\lambda_1,\lambda_2)$ such that the waveguide supports only one propagating mode  and we set  $K:=(\lambda_1,\lambda_2)$}. For later purposes we  note that the eigenfunction $\psi_1$ can be chosen as real-valued { with the following explicit expression 
\begin{equation*}
\psi_1(x_\perp)=
\left\{
\begin{array}{cccc}
\sqrt{ 2/|\Sigma|}\sin( \pi x_\perp/|\Sigma|)  &    \mbox{Dirichlet}   \\
 \sqrt{ 1/|\Sigma|}  &    \mbox{Neumann}  \\
\sqrt{ 2/|\Sigma| } \cos(\pi x_\perp /(2|\Sigma|))  &     \qquad \Gamma_{1\Sigma}=\{x_\perp: x_\perp = |\Sigma|\}\\
 \sqrt{  2/|\Sigma| } \sin(\pi x_\perp /(2|\Sigma|))  &     \qquad \Gamma_{1\Sigma}=\{x_\perp: x_\perp = 0\}
\end{array}
\right.
.
\end{equation*}
}

We now give the definition of the radiation condition.
\begin{definition}
We say a solution to the Helmholtz equation \eqref{ForwardUnbounded1} satisfies the radiation condition if 
\begin{equation}
u^s(x;k) = \sum_{n=1}^{N(k)} a_n e^{i\sqrt{k^2-\lambda_n^2} |x_1|} \psi_n(x_\perp) + \sum_{n=N(k)+1}^\infty a_n e^{-\sqrt{\lambda_n^2-k^2} |x_1|} \psi_n(x_\perp), \quad |x_1| \gg 1,
\end{equation}
where $N(k)$ is the number of the propagating modes, and $a_n$ are coefficients   determined by $u^s(x;k)$.

\end{definition}
This definition would ensure that the forward problem \eqref{ForwardUnbounded1}--\eqref{ForwardUnbounded3} has at most one solution, see for instance \cite{BORCEA2019556,bourgeois2008linear}.

\subsection{Dispersion relation}
For each propagating mode $e^{\pm i\sqrt{k^2-\lambda_n^2} x_1} \psi_n(x_\perp)$, we identify the concept of group wave number by
$$
\mu_n(k) := \sqrt{k^2-\lambda_n^2},
$$
and we call the set $\{(k, \mu_n(k)): k \in (\lambda_n,+\infty)\}$ the $n$-th ``dispersion relation", which describes how this propagating mode propagates with respect to a given  wave number $k$. To illustrate the dispersion relation, in Figure \ref{dispersion1} we plot $\{(k, \mu_1(k)): k \in (1,2)\}$ for a two-dimensional waveguide $(-\infty,\infty) \times (0,\pi)$ with Dirichlet boundary condition.  {For sake of completeness, we also define $\mu_n(k) := \sqrt{k^2-\lambda_n^2}$ (where the chosen branch cut of the square root is with positive imaginary part) for each evanescent modes when $k<\lambda_n$.}

This dispersion relation is a unique feature of waveguide (which is in contrast to the linear dispersion appeared in the far-field patterns in an homogeneous background) and poses difficulty in the single mode multi-frequency sampling method.
%     \begin{figure}[ht!]
%\includegraphics[width=0.46\linewidth]{figures/Dispersion_all}
%     \caption{
%     \linespread{1}
%The plot of $\{(k, \mu_n(k)): k \in (1,6)\}_{n=1}^{n=6}$ for a two-dimensional waveguide $(-\infty,\infty) \times (0,\pi)$ with Dirichlet boundary condition.
%     } \label{dispersion1}
%    \end{figure}
     \begin{figure}[ht!]
\includegraphics[width=0.46\linewidth]{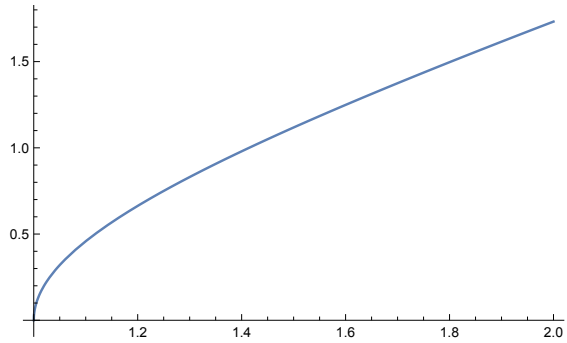}
     \caption{
     \linespread{1}
The plot of $\{(k, \mu_1(k)): k \in (1,2)\}$ for a two-dimensional waveguide $(-\infty,\infty) \times (0,\pi)$ with Dirichlet boundary condition.
     } \label{dispersion1}
    \end{figure}
\subsection{Green function and forward problem}
Now we introduce the Green function $G(x,y;k)$  for the waveguide $W$ that solves
\begin{eqnarray}
\Delta_x G(\cdot,y;k) + k^2 G(\cdot,y;k) = -\delta_y \quad &&\mbox{in} \quad W, \label{Green eqn 1}\\
\mathcal{B} G(\cdot,y;k) =0 \quad &&\mbox{on} \quad \partial W, \label{Green eqn 2}\\
G(\cdot,y;k) \quad  &&\mbox{satisfies the radiation condition}, \label{Green eqn 3}
\end{eqnarray}
here the derivatives are with respect to variable $x$.  The Green function has the following  series expansion \cite{BORCEA2019556,bourgeois2008linear}
\begin{equation} \label{waveguide Green function}
G(x,y;k)=\sum_{n=1}^\infty \frac{i}{2\mu_n(k)} \psi_n(x_\perp)\psi_n(y_\perp)e^{i\mu_n(k) |x_1-y_1|}.
\end{equation}

%The forward problem can be solved by volume integral equation methods. 
Since the forward problem has at most one unique solution (see for instance \cite{BORCEA2019556,bourgeois2008linear}), then the following lemma follows directly by verifying $u^s$  \eqref{lemma eqn us IE} is a solution to the {forward problem} via the {volume} integral, see for instance \cite[Theorem 2.1]{colton2012inverse}. 
\begin{lemma}
There exists a unique solution $u^s \in \widetilde{H}^1_{loc}(W)$ to the {forward} problem \eqref{ForwardUnbounded1}--\eqref{ForwardUnbounded3}. Furthermore it holds that
\begin{equation} \label{lemma eqn us IE}
u^s(\cdot;k) = \int_D G(\cdot,y;k) f(y) \ind y.
\end{equation}
\end{lemma}
 
%-------------------------------------

\section{The multi-frequency far-field operator and its factorization} \label{section operator}
The imaging methods under consideration are the qualitative methods or the sampling methods \cite{cakoni2016qualitative,cakoni2016inverse,cakoni2011linear,colton2012inverse}, and the  goal is to lay the  theoretical foundation for extended sources and design efficient imaging algorithms with these measurements. To achieve this, we introduce the following multi-frequency far-field operator and study its properties in this section.

To begin with, set $\omega_{\sigma\gamma}:=\sqrt{\lambda_1^2+(\sigma-\gamma)^2}$ ({ and details will be given in later context to motivate its definition}).
We introduce the single mode multi-frequency far-field operator $F:L^2(k_-,k_+) \to L^2(k_-,k_+)$ by
%\begin{equation} \label{definition far-field operator}
%(F g)(\sigma;x) :=  \int_{k_-}^{k_+} \frac{\mu_1(\omega_{\sigma\gamma})}{i\psi_1(x_\perp)} \Big[H(\sigma-\gamma) e^{i\theta} u_p^s\left(x; \omega_{\sigma\gamma}\right)- H(\gamma-\sigma) e^{-i\theta}\overline{u_p^s\left(x; \omega_{\sigma\gamma}\right)} \Big] g(\gamma) \ind \gamma,
%\end{equation}
{
\small
\begin{equation} \label{definition far-field operator}
(F g)(\sigma;x^*) :=  \int_{k_-}^{k_+} -i\mu_1(\omega_{\sigma\gamma})\Big[H(\sigma-\gamma) e^{i\theta} u_p^s\left(x^*; \omega_{\sigma\gamma}\right)- H(\gamma-\sigma) e^{-i\theta}\overline{u_p^s\left(x^*; \omega_{\sigma\gamma}\right)} \Big] g(\gamma) \ind \gamma,
\end{equation}
}
$\sigma \in (k_-,k_+)$, where $u_p^s$ represents the propagating part of the wave field $u^s$, $H$ is the Heaviside function
$$
H(s)=
\Big\{
\begin{array}{ccc}
1 ,  &    s\ge0   \\
0 ,  &    s<0  
\end{array},
$$
{ $k_- =0$ and $k_+= \sqrt{\lambda_2^2-\lambda_1^2}$  }  is such that $\omega_{\sigma \gamma} \in (\lambda_1,\lambda_2), \forall \sigma \not= \gamma$. The assumption made on $f$ is   that $e^{i\theta}f>c_1>0$ or $e^{i\theta}f<c_2<0$ with some $\theta \in [0,2\pi)$. {We remark that this assumption in the ``backscattering'' case can be relaxed if both ``backward and forward scattering'' measurements are available (see Section \ref{multi-frequency operator backward and forward measurements}), or we can design another less obvious  multi-frequency operator (see Section \ref{section multi-frequency operator alpha}). We consider this case to fully illustrate the multi-frequency sampling methods in waveguide}. Another note is that $\psi_1({{x}^*_\perp})\not=0$ since otherwise the measurement set \eqref{measurements} would be a vanishing one.

%As can be seen, the possibly non-linear dispersion relation is incorporated into the definition of the far-field operator. This is in contrast to the inverse source problem in the homogeneous whole space.
%In the above definition of $F$, the integral is understood in the Lesbegue sense. 

{As can be seen, the possibly non-linear dispersion relation is incorporated into the definition of the far-field operator { via $\omega_{\sigma\gamma}$}. In particular, one finds via a direct calculation that $\mu_1(\omega_{\sigma\gamma}) = {|\sigma-\gamma|}$ which would allow us to investigate the factorization method further.} This is in contrast to the inverse source problem in the homogeneous whole space where the dispersion is already linear.

To deploy the factorization method based on the multi-frequency operator \eqref{definition far-field operator},  we first introduce the operator $S: L^2(k_-,k_+) \to L^2(D)$ by
\begin{equation} \label{operator S}
(S \varphi)(y) := \int_{k_-}^{k_+} \sqrt{\psi_1(y_\perp)}e^{-i \sigma (y_1-x^*_1)} \varphi(\sigma)  \ind \sigma, \quad y \in D,
\end{equation}
and  the operator $T: L^2(D) \to L^2(D)$ by
%\begin{equation} \label{operator T}
%(T h)(y) :=  \frac{e^{i\theta}f(y)h(y)  }{2}   , \quad y \in D.
%\end{equation}
{
\begin{equation} \label{operator T}
(T h)(y) :=  \frac{e^{i\theta} \psi_1(x^*_\perp)f(y)h(y)  }{2}   , \quad y \in D.
\end{equation}
}
We can directly derive the adjoint operator of $S$, namely $S^*: L^2(D) \to L^2(k_-,k_+)$ by
\begin{equation} \label{operator S*}
(S^* h)(\sigma) :=  \int_D  \overline{\sqrt{\psi_1(y_\perp)}}e^{i \sigma (y_1-x^*_1)} h(y)  \ind y, \quad \sigma \in (k_-,k_+).
\end{equation}

{ We have chosen to keep the conjugate in \eqref{operator S*} to {\it formally} retain the structure of the adjoint.  Note that $\psi_1>0$ a.e. on the cross section, there is no confusion in the  use of square root in \eqref{operator S} and \eqref{operator S*}.}
In the following, we give a  factorization of the far-field operator. 
{Note that $\psi_1>0$  is used in the proof of the following theorem as well.}

\begin{theorem} \label{theorem factorization S*TS}
The far-field operator $F:L^2(k_-,k_+) \to L^2(k_-,k_+)$ given by \eqref{definition far-field operator} can be factorized by 
\begin{equation}
F = S^*TS,
\end{equation}
where $S: L^2(k_-,k_+) \to L^2(D)$ , $T: L^2(D) \to L^2(D)$, and $S^*: L^2(D) \to L^2(k_-,k_+)$ are given by \eqref{operator S},  \eqref{operator T}, and  \eqref{operator S*} respectively.
\end{theorem}
\begin{proof}
Set  $G_p$ the propagating part of the Green function $G$.
From the expression \eqref{lemma eqn us IE} of $u^s$ {and \eqref{waveguide Green function}}, it follows that, for $\sigma \not=\gamma$, the propagating part  $u_p^s$ reads
\begin{eqnarray*}
u^s_p({x^*};\omega_{\sigma\gamma}) &=& \int_D G_p\left({x^*},y; \omega_{\sigma\gamma}\right) f(y) \ind y \nonumber \\
&=& \frac{i}{2\mu_1(\omega_{\sigma\gamma})} \int_D \psi_1({x^*_\perp})\psi_1(y_\perp)e^{i\mu_1(\omega_{\sigma\gamma}) |x^*_1-y_1|}  f(y) \ind y, \quad \sigma \not=\gamma,\nonumber
\end{eqnarray*}
by noting that $\mu_1(\omega_{\sigma\gamma}) = |\sigma-\gamma|$, we then continue to derive that
\begin{eqnarray}
u^s_p({x^*};\omega_{\sigma\gamma}) &=& \frac{i}{2 |\sigma-\gamma| } \int_D \psi_1(x^*_\perp)\psi_1(y_\perp)e^{i |\sigma-\gamma|  |x^*_1-y_1|}  f(y) \ind y, \quad \sigma \not=\gamma,\nonumber
\end{eqnarray}
and consequently (note that $\psi_1$ and $e^{i\theta}f$ are real-valued)
\begin{eqnarray*}
&&\Big[H(\sigma-\gamma)e^{i\theta}u_p^s\left(x^*; \omega_{\sigma\gamma}\right)- H(\gamma-\sigma)e^{-i\theta}\overline{u_p^s\left(x^*; \omega_{\sigma\gamma}\right)}  \Big] \\
=&& \frac{ie^{i\theta}}{2 |\sigma-\gamma| } \int_D \psi_1(x^*_\perp)\psi_1(y_\perp)e^{i (\sigma-\gamma)  |x^*_1-y_1|}  f(y) \ind y, \quad \sigma \not=\gamma.
\end{eqnarray*}
Now from the definition of $F$, it follows that
\begin{eqnarray*}
&&(F g)(\sigma;x^*) =   
%\int_{k_-}^{k_+} \frac{\mu_1(\omega_{\sigma\gamma})}{i\psi_1(x_\perp)}\Big[H(\sigma-\gamma)e^{i\theta}u_p^s\left(x; \omega_{\sigma\gamma}\right)- H(\gamma-\sigma)e^{-i\theta}\overline{u_p^s\left(x; \omega_{\sigma\gamma}\right)} \Big] g(\gamma) \ind \gamma \\
\int_{k_-}^{k_+}  { -i } \mu_1(\omega_{\sigma\gamma}) \Big[H(\sigma-\gamma)e^{i\theta}u_p^s\left(x^*; \omega_{\sigma\gamma}\right)- H(\gamma-\sigma)e^{-i\theta}\overline{u_p^s\left(x^*; \omega_{\sigma\gamma}\right)} \Big] g(\gamma) \ind \gamma \\
&=&1/2\int_{k_-}^{k_+} \int_D  \psi_1(y_\perp)e^{i (\sigma-\gamma) |x^*_1-y_1|}  e^{i\theta} { \psi_1(x^*_\perp)}  f(y) g(\gamma)  \ind y \ind \gamma \\
&=&     \int_D e^{i \sigma |x^*_1-y_1|} \overline{\sqrt{\psi_1(y_\perp)} }\left[\left(\int_{k_-}^{k_+}  \sqrt{\psi_1(y_\perp)}e^{-i \gamma |x^*_1-y_1|} g(\gamma)  \ind \gamma \right) \frac{e^{i\theta} { \psi_1(x^*_\perp)} f(y)  }{2}  \right] \ind y,
\end{eqnarray*}
where in the last step, we used the property that $\psi_1>0$ a.e. on the cross section so that $\sqrt{\psi_1(y_\perp)}=\overline{\sqrt{\psi_1(y_\perp)}}$.
The proof is then completed by recalling the definition of $S$, $T$, and $S^*$ in \eqref{operator S},  \eqref{operator T}, and  \eqref{operator S*} respectively.
\end{proof}

In the following, we study the properties of the factorized operators which are needed to investigate the factorization method in the next section.
\begin{theorem} \label{theorem T coercivity}
The operator $T: L^2(D) \to L^2(D)$ given by  \eqref{operator T} is self adjoint and coercive, and
$$
|\langle Th,h\rangle_{L^2(D)}| \ge c\|h\|^2_{L^2(D)},
$$
for some positive constant $c$.
\end{theorem}
\begin{proof}
Recall that $f\in L^\infty(D)$ with $e^{i\theta}f>c_1>0$ or $e^{i\theta}f<c_2<0$ with some $\theta \in [0,2\pi)$, then $T$ is self adjoint; moreover from the definition of $T$ in \eqref{operator T}
\begin{eqnarray*}
\left|\langle Th,h\rangle_{L^2(D)} \right| &=& \left|\int_D \frac{e^{i\theta}f(y) { \psi_1(x^*_\perp) }|h(y)|^2  }{2}  \ind y \right|,
%=\left|i e^{-i\theta} \int_D \frac{e^{i\theta}f|h(y)|^2  }{2}  \ind y \right|\\
%&=& \left| \int_D \frac{ e^{i\theta}f|h(y)|^2  }{2}  \ind y \right|,
\end{eqnarray*}
then we can directly prove the theorem by noting that $e^{i\theta}f>c_1>0$ or $e^{i\theta}f<c_2<0$ {(and note that  $\psi_1(x^*_\perp) \not =0$ as the measurements are non-vanishing)}. This completes the proof.
\end{proof}

\begin{lemma} \label{lemma S* compact}
The operator $S^*: L^2(D) \to L^2(k_-,k_+)$ given by  \eqref{operator S*}  is compact and has dense range in $L^2(k_-,k_+)$.
\end{lemma}
\begin{proof}
From the definition \eqref{operator S*} of $S^*$, for any $h \in L^2(D)$ we have that
\begin{equation*} 
(S^* h)(\sigma) :=  \int_D  \overline{\sqrt{\psi_1(y_\perp)}}e^{i \sigma (y_1-x^*_1)} h(y)  \ind y, \quad \sigma \in (k_-,k_+),
\end{equation*}
and since $h\in L^2(D)$, one can directly derive that $(S^*h)(\sigma) \in H^1(k_-,k_+)$ by taking the derivative of $(S^* h)(\sigma)$ with respect to $\sigma$. Then we can conclude the compactness of $S^*: L^2(D) \to L^2(k_-,k_+)$ from the compact embedding from $H^1(k_-,k_+)$ into $L^2(k_-,k_+)$.  

{

To show the dense range of $S^*$ in $L^2(k_-,k_+)$, it is sufficient to show that $S$ is injective. We prove this by contradiction. Indeed if there exists $\varphi \in L^2(k_-,k_+)$ such that $S\varphi=0$, i.e.,
$$
\int_{k_-}^{k_+} \sqrt{\psi_1(y_\perp)}e^{-i \sigma (y_1-x^*_1)} \varphi(\sigma)  \ind \sigma =0, \quad \forall y \in D,
$$
note that $\psi_1\not=0$, then we can derive that 
$$
\mathcal{F}\{ \varphi \chi_{(k_-,k_+)} \} (y_1-x^*_1) =0, \quad \forall y_1 \in D_1^+,
$$
where $\mathcal{F}$ is the Fourier transform $\mathcal{F}:\varphi \chi_{(k_-,k_+)}  \to \int_{\mathbb{R}} \varphi(y) \chi_{(k_-,k_+)}(y)  e^{-is  y} \ind y$,  $D^+_1:=\{y_1: \exists\, y_\perp^* \mbox{ s.t. } (y_1,y_\perp^*) \in D\}$ and $\chi_{(k_-,k_+)}$ is the  characteristic function such that $\chi_{(k_-,k_+)}(x)=1$ if $x\in (k_-,k_+)$ and $\chi_{(k_-,k_+)}(x)=0$ if $x \not\in (k_-,k_+)$. This implies that $\varphi \chi_{(k_-,k_+)}$ vanishes and thereby  $\varphi$ vanishes in $L^2(k_-,k_+)$. This proves that $S$ is injective and hence  $S^*$   has dense range in $L^2(k_-,k_+)$. This completes the proof.
}
\end{proof}

%%%%%%
\section{The sampling methods} \label{section sampling methods}
In this section, we study the  sampling methods \cite{cakoni2016qualitative,cakoni2016inverse,cakoni2011linear,colton2012inverse,kirsch2008factorization} where the goal is to lay the  theoretical foundation for extended sources and design efficient imaging algorithms with the single mode multi-frequency measurements. The factorization method  \cite{Kirsch98} has established a mathematically rigorous characterization of extended scatterers in the classical setting and has been applied to the inverse source problem in the homogeneous whole space \cite{GriesmaierSchmiedecke-source} with sparse multi-frequency measurements. There have been relevant work on the factorization method in waveguide with a single frequency, yet we aim to deploy the factorization method with a single propagating mode at multiple frequencies. Based on the multi-frequency operator and its properties derived in Section \ref{section operator}, the factorization method is to be established in Section \ref{section factorization method}. Moreover, the factorization of the multi-frequency operator may also lead to a factorization-based sampling method, as demonstrated in waveguide imaging in the single frequency case \cite{borcea2019factorization}
 and the inverse source problem in the homogeneous space \cite{LiuMeng2021} with sparse multi-frequency measurements. 
 In Section \ref{section factorization based sampling method} the factorization-based sampling method is to be analyzed. \subsection{The factorization method}\label{section factorization method}
In this section, we shall investigate the factorization method which gives a mathematically rigorous characterization of the range support of the source. To begin with, denote by
\begin{equation} \label{def psiz epsilon}
\psi_z^\epsilon(\sigma) := \frac{1}{|B(z,\epsilon)|}\int_{B(z,\epsilon)}e^{i\sigma(y_1-x^*_1)}  \ind y,
\end{equation}
where $B(z,\epsilon)$ is a circle  centered at $z$ with radius $\epsilon$, and $|B(z,\epsilon)|$ denotes its area.

We first give a characterization of the range of $S^*$. Let $D^+:=\{(y_1,y_\perp)\in W: \exists\, y_\perp^* \mbox{ s.t. } (y_1,y_\perp^*) \in D\}$, and recall that $D^+_1=\{y_1: \exists\, y_\perp^* \mbox{ s.t. } (y_1,y_\perp^*) \in D\}$.  For any domain $\Omega$, denote by $\chi_\Omega$ the {characteristic} function such that $\chi_\Omega(x)=1$ if $x\in \Omega$ and $\chi_\Omega(x)=0$ if $x \not\in \Omega$. For any point $z=(z_1,z_\perp) \in W$, we call $z_1$ the range coordinate and $z_\perp$ the cross-section coordinate.
\begin{lemma} \label{lemma range S*}
Let $S^*: L^2(D) \to L^2(k_-,k_+)$ be given by \eqref{operator S*} and $\psi_z^\epsilon$ be given by \eqref{def psiz epsilon} respectively.
For any sampling point  in $W$ with range coordinate $z_1$, it holds that
\begin{enumerate}
\item
If $z_1 \in D^+_1$, then there exists a positive $\epsilon$ depending on $z_1$ such that for any $z$ with range coordinate $z_1$,  $\psi_{z}^\epsilon \in Range(S^*)$. 
\item
If $z_1 \not \in D^+_1$, then for any sufficiently small $\epsilon>0$ with {$[z_1-\epsilon,z_1+\epsilon]  \cap \overline{D_1^+}=\emptyset$} and any $z$ with range coordinate $z_1$, $\psi_z^\epsilon \not\in Range(S^*)$.
\end{enumerate}
\end{lemma}
\begin{proof}
(1). If ${z_1} \in D^+_1$, then there must exist $\epsilon>0$ and $z^*=(z_1,z^*_\perp) \in D$ such that $B(z^*,\epsilon) \in D$. Let $h(y)=  \frac{\chi_{B(z^*,\epsilon)}(y) }{|B(z^*,\epsilon)|} \frac{1}{\overline{\sqrt{\psi_1(y_\perp)}}} $, then $h\in L^2(D)$ and
it is directly verified that 
\begin{equation*}
(S^* h)(\sigma)  =  \int_D   \frac{\chi_{B(z^*,\epsilon)}(y) }{|B(z^*,\epsilon)|}   e^{i \sigma (y_1-x^*_1)}  \ind y = \psi_{z^*}^\epsilon (\sigma).
\end{equation*}
This proves that $\psi_{z^*}^\epsilon \in Range(S^*)$. Now a direct calculation
{
\begin{eqnarray*}
\psi_{z^*}^\epsilon(\sigma) &=& \frac{1}{|B(z^*,\epsilon)|}\int_{B(z^*,\epsilon)}e^{i\sigma(y_1-x^*_1)}  \ind y = \frac{1}{ \epsilon^2}\int_{z_1-\epsilon}^{z_1+\epsilon} \int_{z^*_\perp-\sqrt{\epsilon^2-(y_1-z_1)^2}}^{z^*_\perp+\sqrt{\epsilon^2-(y_1-z_1)^2}}     e^{i\sigma(y_1-x^*_1)}  \ind y_\perp  \ind y_1 \\
&=&  \frac{1}{\epsilon^2}\int_{z_1-\epsilon}^{z_1+\epsilon}   2 \sqrt{\epsilon^2-(y_1-z_1)^2}   e^{i\sigma(y_1-x^*_1)}   \ind y_1
\end{eqnarray*}}
yields that $\psi_{z^*}^\epsilon$ is indeed independent of the cross-section coordinate $z^*_\perp$, i.e. for any $z=(z_1,z_\perp)$ and $z^*=(z_1,z^*_\perp)$, $\psi_{z}^\epsilon= \psi_{z^*}^\epsilon \in Range(S^*)$. This proves the first part.

(2). If $z_1 \not \in {D_1^+}$, for any sufficiently small $\epsilon>0$ with $[z_1-\epsilon,z_1+\epsilon]  \cap \overline{D_1^+}=\emptyset$, assume on the contrary that $\psi_{z}^\epsilon \in Range(S^*)$ for a $z=(z_1,z_\perp)$. Then there exists a $g \in L^2(D)$ such that
\begin{equation*}
\int_D  \overline{\sqrt{\psi_1(y_\perp)}} e^{i \sigma (y_1-x^*_1)} g(y)  \ind y = (S^* g)(\sigma)  = \psi_{z}^\epsilon (\sigma) =  \frac{1}{|B(z,\epsilon)|}\int_{B(z,\epsilon)}e^{i\sigma(y_1-x^*_1)}  \ind y,
\end{equation*}
for any $\sigma \in (k_-,k_+)$. Noting the common term $e^{i\sigma x^*_1}$, the above equation is reduced to 
\begin{equation} \label{lemma range S* proof eqn1}
\int_{\mathbb{R}^d}  e^{i \sigma y_1} \chi_{D} (y)\widetilde{g}(y)  \ind y =  \int_{\mathbb{R}^d}   e^{i\sigma y_1}    \frac{\chi_{B(z,\epsilon)}(y)}{|B(z,\epsilon)|} \ind y, \qquad \forall \sigma \in (k_-,k_+)
\end{equation}
where $ \widetilde{g}(y)  :=\overline{\sqrt{\psi_1(y_\perp)}}  g(y)$ in $D$. 

%We can rewrite \eqref{lemma range S* proof eqn1} as
%\begin{equation} \label{lemma range S* proof eqn2}
%\int_{\mathbb{R} } e^{i \sigma y_1}  \int_{\mathbb{R}^{d-1}}   \chi_{D} (e_1y_1 + e_\perp y_\perp) \widetilde{g}(e_1y_1 + e_\perp y_\perp)  \ind y_\perp \ind y_1 =  \int_{\mathbb{R}^d}   e^{i\sigma y_1}    \frac{\chi_{B(z,\epsilon)}(y)}{|B(z,\epsilon)|} \ind y, \qquad \forall \sigma \in (k_-,k_+)
%\end{equation}

By recalling the multi-dimensional Fourier transform $\mathcal{F}: L^2(\mathbb{R}^d) \to L^2(\mathbb{R}^d)$
\begin{equation}
 \mathcal{F}\{u\} (s) = \int_{\mathbb{R}^d} u(y) e^{-is\cdot y} \ind y, \quad s\in \mathbb{R}^d, \quad \forall u \in L^2(\mathbb{R}^d),
\end{equation}
and the {Radon} transform $\mathcal{R}^e: L^2(\mathbb{R}^d) \to L^1(\mathbb{R} )$ for a unit direction $e\in \mathbb{R}^d$ that
\begin{equation}
 \mathcal{R}^e \{u\} (t')  =\int_{e^\perp} u(t'e+{y'})  \ind y', \quad t' \in \mathbb{R} , \quad \forall u \in L^2(\mathbb{R}^d),
\end{equation}
where $e^\perp$ represents the hyperplane orthogonal to $e$. {From the Fourier slice theorem \cite[Theorem 1.1]{natterer2001book}}
\begin{equation}
 \mathcal{F}\{u\} (te) = \widehat{ \mathcal{R}^e}\{u\}(t),
\end{equation}
where $~\widehat{}~$ denotes the one-dimensional Fourier transform.

Now we can derive from \eqref{lemma range S* proof eqn1} that
\begin{equation*} 
 \mathcal{F}\left\{ \chi_{D} \widetilde{g} \right\} ( -\sigma e_1) = \mathcal{F}\left\{ \frac{\chi_{B(z,\epsilon)} }{|B(z,\epsilon)|} \right\} (-\sigma e_1), \quad e_1 = (1,0,\cdots), \qquad \forall \sigma \in (k_-,k_+).
\end{equation*}
and thereby 
\begin{equation}\label{lemma range S* proof eqn3}
\widehat{ \mathcal{R}^{e_1}}\left\{ \chi_{D}  \widetilde{g} \right\}  (-\sigma)= \widehat{\mathcal{R}^{e_1}}\left\{ \frac{\chi_{B(z,\epsilon)} }{|B(z,\epsilon)|} \right\}  (-\sigma), \qquad \forall \sigma \in (k_-,k_+).
\end{equation}
%Since it is directly observed that $~\mbox{supp}~\mathcal{R}^{e_1}\left\{ \chi_{D}  \widetilde{g} \right\} \subset e_1\cdot D = D^+_1$ and $~\mbox{supp}~\mathcal{R}^{e_1}\left\{ \frac{\chi_{B(z,\epsilon)} }{|B(z,\epsilon)|} \right\} \subset e_1\cdot B(z,\epsilon) = (z_1-\epsilon,z_1+\epsilon)$, then the compactly supported functions $\mathcal{R}^{e_1}\left\{ \chi_{D}  \widetilde{g} \right\}$ and $\mathcal{R}^{e_1}\left\{ \frac{\chi_{B(z,\epsilon)} }{|B(z,\epsilon)|} \right\}$ would allow one to derive that $\widehat{ \mathcal{R}^{e_1}}\left\{ \chi_{D}  \widetilde{g} \right\}  $ and $\widehat{\mathcal{R}^{e_1}}\left\{ \frac{\chi_{B(z,\epsilon)} }{|B(z,\epsilon)|} \right\}  $ are analytic in the real line due to the Paley-Wiener Theorem. This shows that equation \eqref{lemma range S* proof eqn3} holds in the real line and consequently
Then we can derive that
\begin{equation*}
\mathcal{R}^{e_1}\left\{ \chi_{D}  \widetilde{g} \right\}  =  \mathcal{R}^{e_1}\left\{ \frac{\chi_{B(z,\epsilon)} }{|B(z,\epsilon)|} \right\}  .
\end{equation*}
%and consequently $~\mbox{supp} ~\mathcal{R}^{e_1}\left\{ \chi_{D}  \widetilde{g} \right\}~\cap~\mbox{~supp}~\mathcal{R}^{e_1}\left\{ \frac{\chi_{B(z,\epsilon)} }{|B(z,\epsilon)|} \right\} \not=\{0\}$.
However $~\mbox{supp}~\mathcal{R}^{e_1}\left\{ \chi_{D}  \widetilde{g} \right\} \subset e_1\cdot D = D^+_1$ and $~\mbox{supp}~\mathcal{R}^{e_1}\left\{ \frac{\chi_{B(z,\epsilon)} }{|B(z,\epsilon)|} \right\} \subset e_1\cdot B(z,\epsilon) = (z_1-\epsilon,z_1+\epsilon)$ yields  a contradiction (since $[z_1-\epsilon,z_1+\epsilon]  \cap \overline{D_1^+}=\emptyset$). This prove the second part. These prove the Lemma.
%, and thereby
%\begin{equation}
%\int_{\mathbb{R}^d}  e^{i \sigma y_1}  \Big(\chi_{D} (y)\widetilde{g}(y) - \frac{\chi_{B(z,\epsilon)}(y)}{|B(z,\epsilon)|} \Big) \ind y =0.
%\end{equation}
\end{proof}

As can be seen from Lemma \ref{lemma range S*} part $(1)$, the $\epsilon$ depends on the  sampling range coordinate $z_1$. The following corollary removes this dependence with an extra geometric assumption on $D$.

\begin{corollary} \label{corollary range S*}
Assume in addition that the minimal width of $D$ is bounded below by a positive constant, i.e., {$\min\{ L(y_1) : \forall y_1 \}>c>0$ where $L(y_1):={ \min_{y_\perp}\max_{y_\perp^*}}\{|y_\perp - y_\perp^*|: (y_1, \beta y_\perp+(1-\beta) y_\perp^*) \in D, \forall \beta \in[0,1] \}$}. For any sampling point in $W$ with range coordinate $z_1$, it holds that
\begin{enumerate}
\item If $z_1 \in D^+_1$, { let $\epsilon_1>0$ be any given small tolerance level, then there exists a positive $\epsilon_0$ such that for any $z$ with $(z_1-\epsilon_1,z_1+\epsilon_1) \subset D_1^+$ and any positive $\epsilon<\epsilon_0$,} 
%then there exists a positive $\epsilon_0$ independent of $z_1$ such that for any $z$ with range coordinate $z_1$ and any positive $\epsilon<\epsilon_0$,  
$\psi_{z}^\epsilon \in Range(S^*)$. 

\item If $z_1 \not \in D_1^+$, then for any sufficiently small $\epsilon>0$ with $[z_1-\epsilon,z_1+\epsilon]  \cap \overline{D_1^+}=\emptyset$ and any $z$ with range coordinate $z_1$, $\psi_z^\epsilon \not\in Range(S^*)$.
\end{enumerate}
\end{corollary}
\begin{proof}
The only difference is in part (1). We highlight the difference. Assume that $z_1 \in D^+_1$. Since the minimal width of $D$ is bounded below by a positive constant, i.e., ${  \min_{y_\perp}\max_{y_\perp^*}}\{|y_\perp - y_\perp^*|: (y_1, \beta y_\perp+(1-\beta) y_\perp^*) \in D, \forall \beta \in[0,1] \}>c>0$, then for any range coordinate $z_1 \in D^+_1$ { with $(z_1-\epsilon_1,z_1+\epsilon_1) \subset D$}, there must exist $z^*=(z_1,z^*_\perp) \in D$ such that $B(z^*,\epsilon) \in D$ for all $0<\epsilon < { \min\{c/4,\epsilon_1/4\}}$.
Now let $h(y)=  \frac{\chi_{B(z^*,\epsilon)}(y) }{|B(z^*,\epsilon)|}  \frac{1}{\overline{\sqrt{\psi_1(y_\perp)}}} $, then $h\in L^2(D)$ and
it is directly verified that 
\begin{equation*}
(S^* h)(\sigma)  =  \int_D  \frac{\chi_{B(z^*,\epsilon)}(y) }{|B(z^*,\epsilon)|}  e^{i \sigma (y_1-x^*_1)}  \ind y = \psi_{z^*}^\epsilon (\sigma).
\end{equation*}
This proves that $\psi_{z^*}^\epsilon \in Range(S^*)$. Using the same argument before, we prove the corollary.
\end{proof}

Applying Theorem \ref{theorem T coercivity} and Lemma \ref{lemma S* compact}, the following Lemma is a direct application of \cite[Corollary 1.22]{kirsch2008factorization}.
\begin{lemma}
Let $F:L^2(k_-,k_+) \to L^2(k_-,k_+)$ be given by \eqref{definition far-field operator} and  $S^*: L^2(D) \to L^2(k_-,k_+)$ be given by \eqref{operator S*} respectively. It holds that
$Range(S^*)=Range(F^{1/2})$.
\end{lemma}

Now we can prove the following theorem. Let $(\phi_j,\alpha_j)_{j=1}^\infty$ be the eigensystem of the self-adjoint operator $F^{1/2}$ {and let $\langle \cdot , \cdot \rangle$ denote the inner product in $ L^2(k_-,k_+)$}.
\begin{theorem}\label{theorem FM main}
For any sampling point in $W$ with range coordinate $z_1$, it holds that
\begin{enumerate}
\item If $z_1 \in D^+_1$, then there exists a positive $\epsilon$ depending on $z_1$  such that for any $z$ with range coordinate $z_1$, $\psi_{z}^\epsilon \in Range(F^{1/2})$ and moreover
\begin{equation}\label{theorem FM eqn1}
\sum_{j=1}^\infty \frac{|\langle \psi_{z}^\epsilon,\phi_j \rangle|^2}{\alpha_j^2} < \infty.
\end{equation}
\item
If $z_1 \not \in D^+_1$, then for any sufficiently small $\epsilon>0$ with $[z_1-\epsilon,z_1+\epsilon]  \cap \overline{D_1^+}=\emptyset$ and any $z$ with range coordinate $z_1$, $\psi_z^\epsilon \not\in Range(F^{1/2})$ and moreover
\begin{equation}\label{theorem FM eqn2}
\sum_{j=1}^\infty \frac{|\langle {\psi_{z}^\epsilon},\phi_j \rangle|^2}{\alpha_j^2} = \infty.
\end{equation}
\end{enumerate}
\end{theorem}
\begin{proof}
Since $Range(S^*)= Range(F^{1/2})$. Then the proof is completed by applying Lemma \ref{lemma range S*} and Picard theorem \cite[Theorem 4.8]{colton2012inverse}.
\end{proof}

\begin{corollary} \label{corollary FM main}
Assume in addition that the minimal width of $D$ is bounded below by a positive constant, i.e.,   {$\min\{ L(y_1) : \forall y_1 \}>c>0$ where $L(y_1)={ \min_{y_\perp}\max_{y_\perp^*}}\{|y_\perp - y_\perp^*|: (y_1, \beta y_\perp+(1-\beta) y_\perp^*) \in D, \forall \beta \in[0,1] \}$}.  Then we have that
\begin{enumerate}
\item If $z_1 \in D^+_1$, 
{  let $\epsilon_1>0$ be any given small tolerance level, then there exists a positive $\epsilon_0$ such that for any $z$ with $(z_1-\epsilon_1,z_1+\epsilon_1) \subset D_1^+$ and any positive $\epsilon<\epsilon_0$,} 
%then there exists a positive $\epsilon_0$ independent of $z_1$ such that, for all $0<\epsilon<\epsilon_0$ and any $z$ with range coordinate $z_1$, 
$\psi_{z}^\epsilon \in Range(F^{1/2})$ and equation \eqref{theorem FM eqn1} holds.
\item If $z_1 \not \in D^+_1$, then for any sufficiently small $\epsilon>0$ with $[z_1-\epsilon,z_1+\epsilon]  \cap \overline{D_1^+}=\emptyset$ and any $z$ with range coordinate $z_1$, $\psi_z^\epsilon \not \in Range(F^{1/2})$ and equation \eqref{theorem FM eqn2} holds.
\end{enumerate}
\end{corollary}

\begin{remark}
Corollary \ref{corollary FM main} implies that the $\epsilon$ can be chosen as a fixed constant for all the sampling points $z$ when implementing the factorization method. In particular, we can choose $\epsilon$ as a  small number such that $\epsilon$ is less than a tolerance level that one is seeking for. If there is no such assumption as in Corollary \ref{corollary FM main}, we can still chose $\epsilon$ as a fixed tolerance level, and in this case, the part whose width is  very small  may not be reconstructed as indicated by Theorem \ref{theorem FM main} and Lemma \ref{lemma range S*} (and its proof). However since $\epsilon$ is the tolerance level, the factorization method theoretically would still yield a sufficiently good image.
\end{remark}
%%%%%%
\subsection{The factorization-based sampling method} \label{section factorization based sampling method}
In this section, we investigate a factorization-based multi-frequency sampling method with  imaging function given by
\begin{equation}
I(z):=| \langle F \psi_z,\psi_z \rangle |, \quad \psi_z(\sigma) = e^{i\sigma(z_1-x^*_1)}.
\end{equation}
%This imaging function makes use of multiple frequencies, and it is worth mentioning that a similar sampling method \cite{LiuMeng2021} is investigated for the inverse source problem in the whole space.
Its property is given by the following theorem.
\begin{theorem} \label{factorization-based SM theorem}
There exists positive constants $c_1$ and $c_2$ such that
\begin{equation}
c_1 \|S \psi_z\|^2_{L^2(D)}  \le |I(z)| \le c_2 \|S \psi_z\|^2_{L^2(D)}.
\end{equation}
\end{theorem}
\begin{proof}
From the factorization of the far-field operator, it first follows that
\begin{eqnarray*}
I(z)= | \langle F \psi_z,\psi_z \rangle_{L^2(k_-,k_+)} | = |\langle S^*TS \psi_z,\psi_z \rangle_{L^2(k_-,k_+)} | = |\langle  TS \psi_z, S\psi_z \rangle_{L^2(D)} |.
\end{eqnarray*}
The lower bound then follows from the coercivity of $T$ in Theorem \ref{theorem T coercivity}, the upper bound follows directly from the boundedness of $T$. This completes the proof of the theorem.
\end{proof}
Theorem \ref{factorization-based SM theorem} implies that $I(z)$ is qualitatively the same as $\|S \psi_z\|^2_{L^2(D)}$ without the knowledge of $D$. {The following Lemma gives the explicit expression of $S \psi_z$}.
\begin{lemma} \label{theorem s psiz formula}
It holds that
\begin{equation} \label{theorem s psiz formula eqn}
(S \psi_z)(y) =  \sqrt{\psi_1(y_\perp)}   \frac{e^{ik_+(z_1-y_1)}-e^{ik_-(z_1-y_1)}}{i(z_1-y_1)}, \quad y \in D,
\end{equation}
where the above quotient is understood in the limit sense when $z_1=y_1$.
\end{lemma}
\begin{proof}
From the definition of $S$, we directly have that
\begin{eqnarray*}
(S \psi_z)(y) &=& \sqrt{\psi_1(y_\perp)}  \int_{k_-}^{k_+}   e^{i\sigma(z_1-y_1)}   \ind \sigma \\
&=& \sqrt{\psi_1(y_\perp)} \frac{e^{ik_+(z_1-y_1)}-e^{ik_-(z_1-y_1)}}{i(z_1-y_1)},
\end{eqnarray*}
where the above quotient is understood in the limit sense when $z_1=y_1$. This completes the proof.
\end{proof}
     \begin{figure}[ht!]
\includegraphics[width=0.5\linewidth]{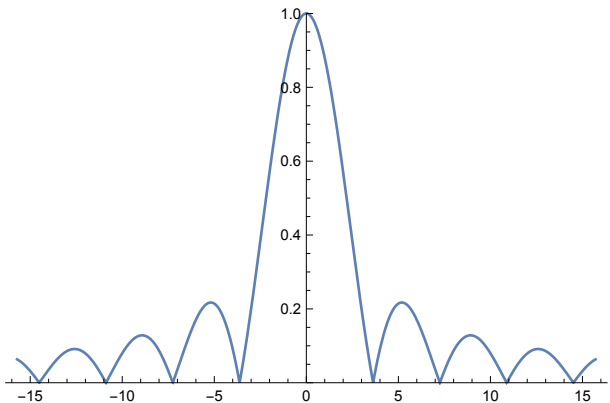}
     \caption{
     \linespread{1}
The plot of $|S\psi_z|/\max\{S\psi_z\}$ with $y_1=0$ for a two-dimensional waveguide $(-\infty,\infty) \times (0,\pi)$ with Dirichlet boundary condition, $k_-=0$ and $k_+=\sqrt{3}$. The horizontal axis is $z_1$ and $z_1 \in (-5\pi,5\pi)$.
     } \label{resolution x1}
    \end{figure}
 
From the explicit formula of $S\psi_z$ in Lemma \ref{theorem s psiz formula}, we find that its maximum is obtained at $z_1=y_1$ and that $S\psi_z$ is oscillatory and decaying to zero when $|z_1-y_1|$ becomes large. We illustrate this property in Figure  \ref{resolution x1} with a fixed $y_1=0$ and $y_\perp$, where we plot $|S\psi_z|/\max\{|S\psi_z|\}$   for a two-dimensional waveguide $(-\infty,\infty) \times (0,\pi)$ with Dirichlet boundary condition where $k_-=0$ and $k_+=\sqrt{3}$. We further remark that the function $S\psi_z$ is only sensitive to the range sampling coordinate $z_1$ which can be directly observed from \eqref{theorem s psiz formula eqn}.  As a summary, {it is heuristically observed that $\|S \psi_z\|^2_{L^2(D)}$ peak in  $D^+=\{(y_1,y_\perp) \in W: \exists\, y_\perp^* \mbox{ s.t. } (y_1,y_\perp^*) \in D\}$ as the sampling point $z$ samples in a sampling region since there always exists $y=(y_1,y_\perp)\in D$ such that $y_1=z_1$ for any $z=(z_1,z_\perp) \in D^+$}. Together with
Theorem \ref{factorization-based SM theorem} it is inferred  that $I(z)$ is qualitatively the same as $\|S \psi_z\|^2_{L^2(D)}$ without the knowledge of $D$.
%Theorem \ref{} implies that $I(z)$ is qualitatively the same as $\|S \psi_z\|^2_{L^2(D)}$; Theorem \ref{} and Remark \ref{} imply that $\|S \psi_z\|^2_{L^2(D)}$ peak in a strip $D^+:=\{(y_1,y_\perp): \exists\, y_\perp^* \mbox{ s.t. } (y_1,y_\perp^*) \in D\}$ including $D$.  
{ Though the factorization-based sampling method is not as mathematically rigorous as the factorization method, 
it is heuristically inferred that $I(z)$ peaks in  $D^+$ which is the range support of $D$ in the waveguide. }

%%%%%%%%%%%%%%%%%%%%%%
\section{Numerical examples} \label{section numerical example}
In this section, we present a variety of numerical examples in two dimensions to illustrate the performance of the single mode multi-frequency sampling methods.

We generate the synthetic data $u^s$ using the finite element computational software Netgen/NGSolve \cite{schoberl1997netgen}. To be more precise, the computational domain is ${(-4,4)}\times \Sigma$ and the measurements are at  a random location on the cross section ${\{-2\}} \times \Sigma$. We apply a brick Perfectly Matched Layer (PML) in the domain $\{x_1: 2.5<|x_1|<4\}\times \Sigma$ and choose PML absorbing coefficient $5+5i$.   In all of the numerical examples, we apply the second order finite element to solve for the  wave field where the source is constant $1$; the mesh size is chosen as $|\Sigma|/10$  where {{we recall that $|\Sigma|$ denotes} the length  of the cross section. We further add $5\%$ Gaussian noise to the synthetic data $u^s$ to implement the imaging function in Matlab.

We first illustrate the implementation of the factorization method and the factorization-based sampling method. Let $\{k_1,k_2,\cdots,k_N\}$ be the set of equally discretized wave numbers, this yields the discretized $N\times N$ multi-frequency far-field matrix $F_N$ of its continuous counterpart $F$ in \eqref{definition far-field operator}.

\begin{enumerate}
\item The factorization method.  Let $(\widehat{\phi}_j,\widehat{\alpha}_j)_{j=1}^N$ be the eigensystem of the self-adjoint matrix $(F_N)^{1/2}$. The imaging function indicated by the factorization method is 
\begin{equation*}
I_{\mbox{\tiny FM}}(z):= \left( \sum_{j=1}^N \frac{|\langle \widehat{\psi}_{z}^\epsilon,\widehat{\phi}_j \rangle|^2}{\widehat{\alpha}_j^2} \right)^{-1},
\end{equation*}
where $\epsilon$ is a fixed tolerance level, and $\widehat{\psi}_{z}^\epsilon$ is the discretization of ${\psi}_{z}^\epsilon$ \eqref{def psiz epsilon} at the multiple wavenumbers (i.e. multi-frequencies) $\{k_1,k_2,\cdots,k_N\}$. The $\langle \cdot, \cdot \rangle$ represents the inner product of two vectors that arise from the discretization of their continuous counterparts. The main Theorem \ref{corollary FM main} of the factorization method implies that $I_{FM}(z)$ is bounded below by $0$ for $z\in D^+_{1,\epsilon}$  and almost vanishes for  $z_1 \not \in D^+_{1,\epsilon}$ where $D^+_{1,\epsilon}$ is a neighborhood of the support $D_1^+$ within the tolerance level $\epsilon$. We choose tolerance level $\epsilon=0.01$ in all of the numerical examples. Though we have fixed this tolerance level, we have tested other small tolerance levels and haven't found any visual difference on the imaging results.
%We also consider the limiting case  $\epsilon \to 0$ which implies that  $\psi_z^\epsilon(\sigma) \to  e^{i\sigma(z_1-x_1)}$.

         \begin{figure}[ht!]
\includegraphics[width=0.49\linewidth]{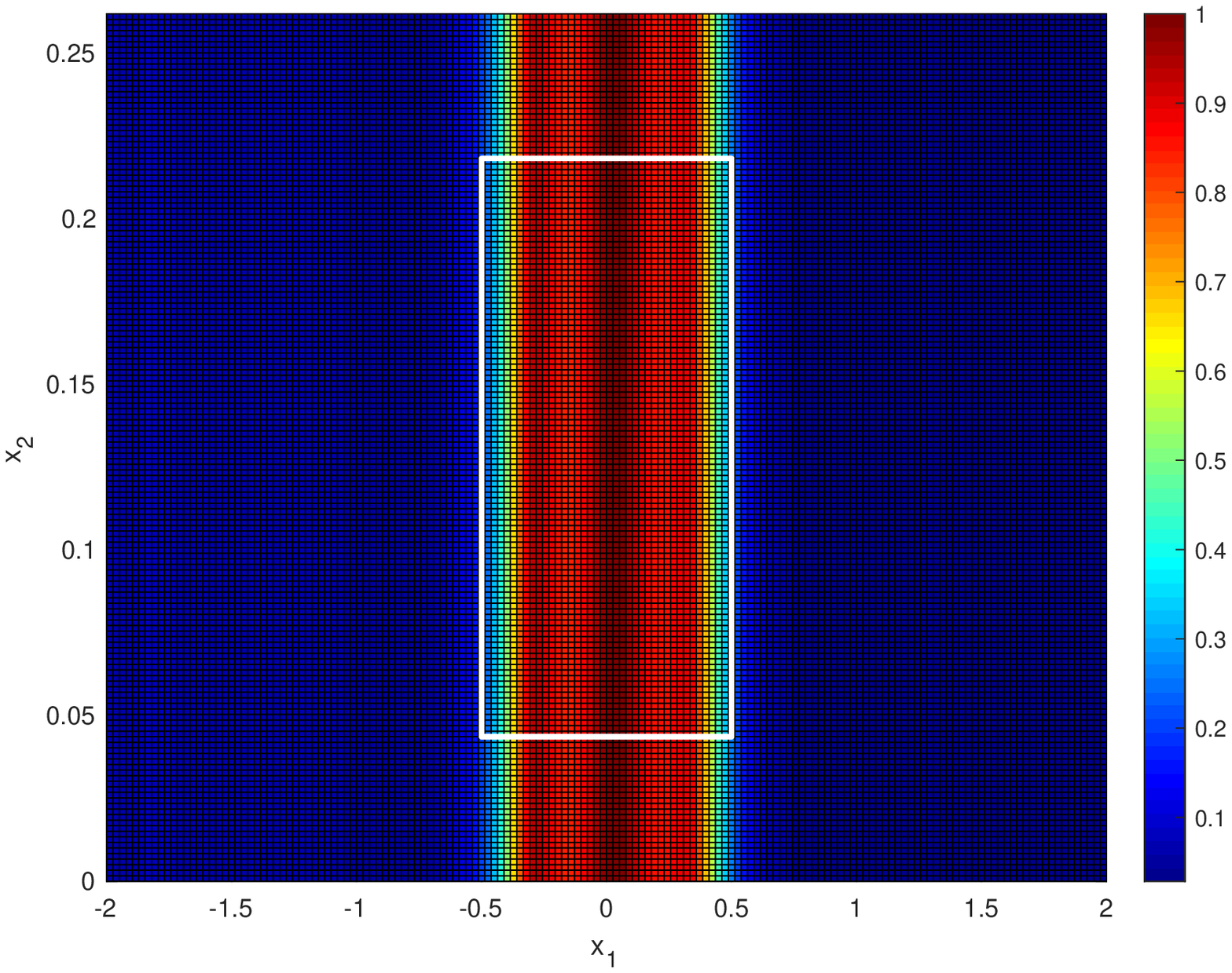}
\includegraphics[width=0.49\linewidth]{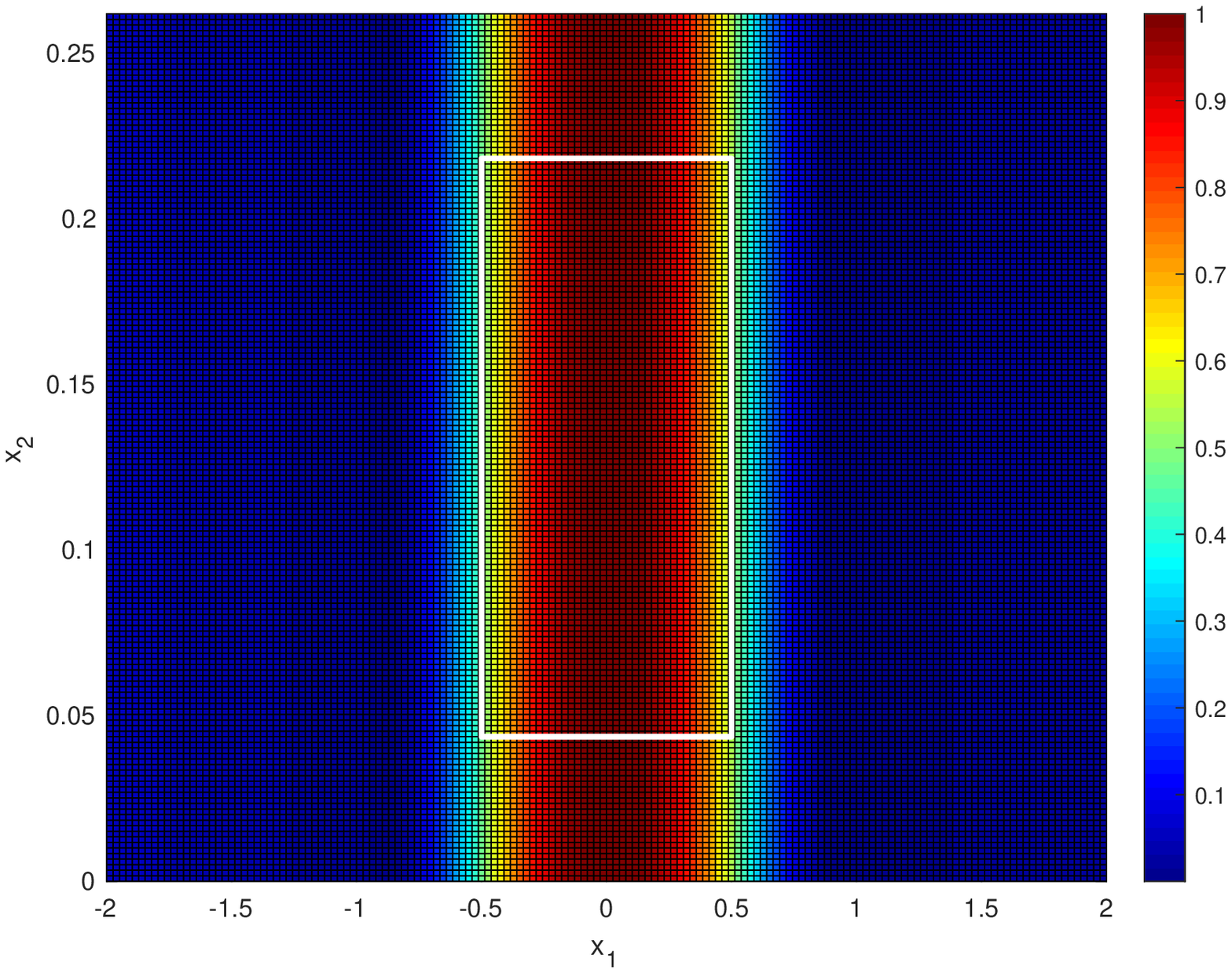}
\includegraphics[width=0.49\linewidth]{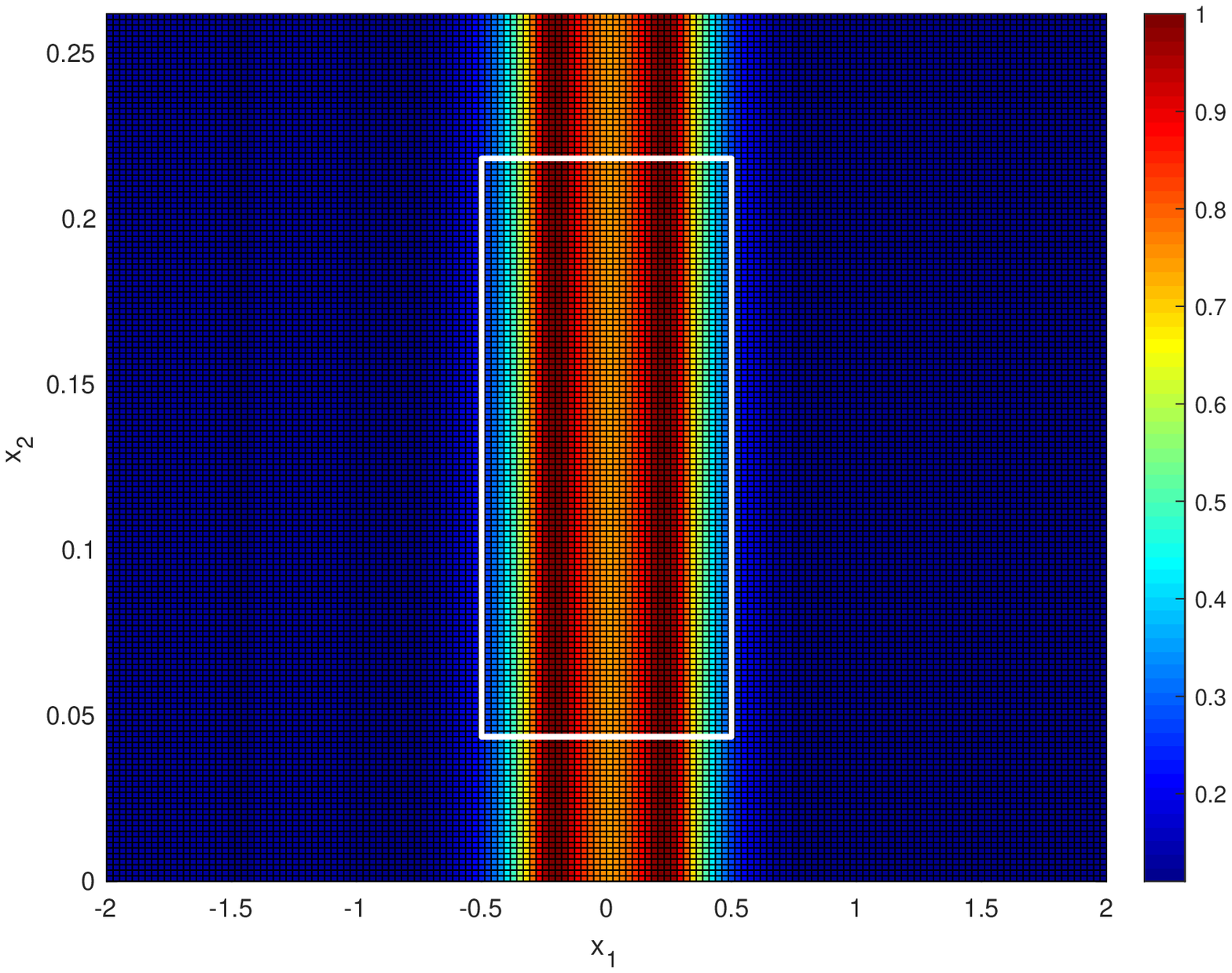}
\includegraphics[width=0.49\linewidth]{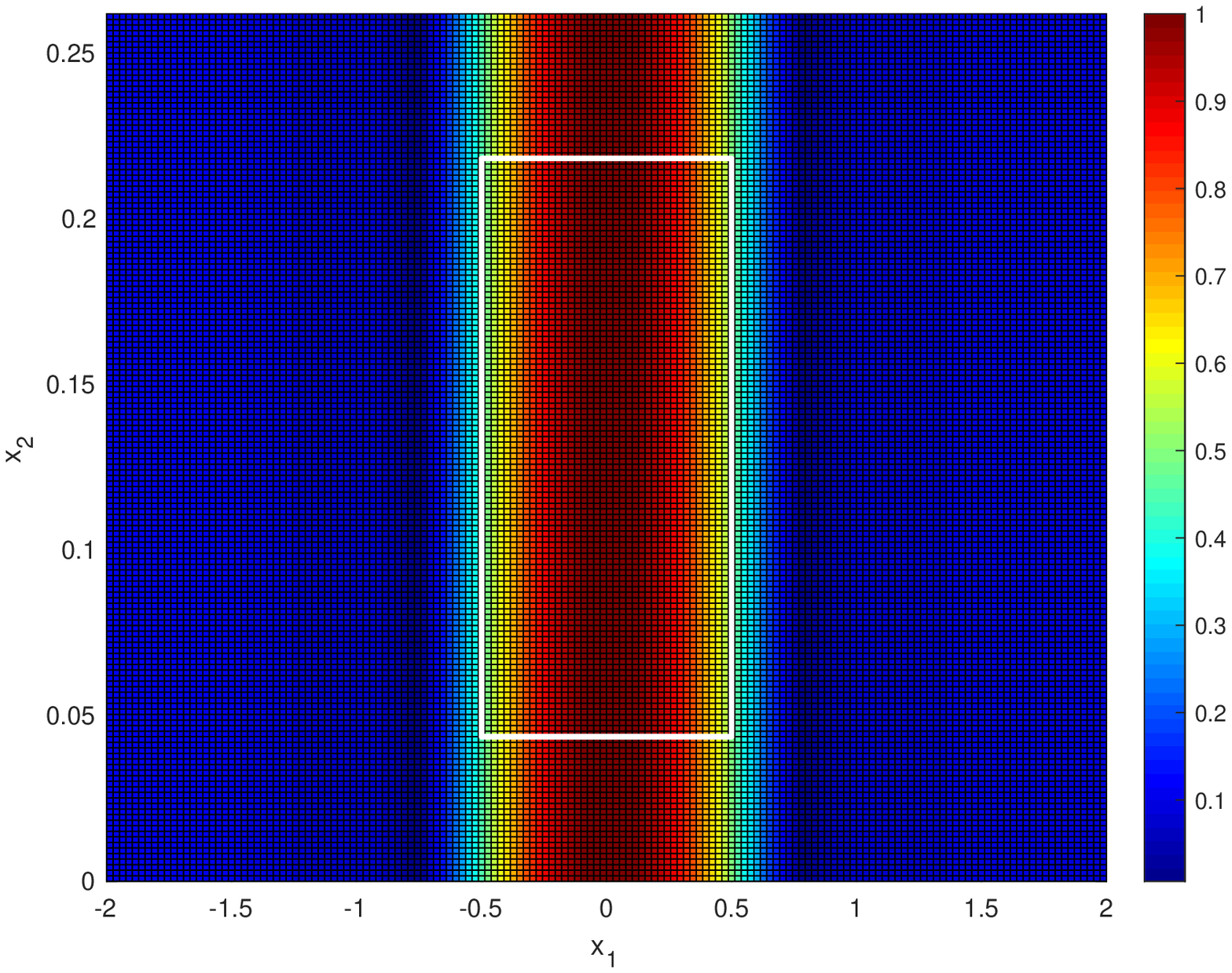}
\includegraphics[width=0.49\linewidth]{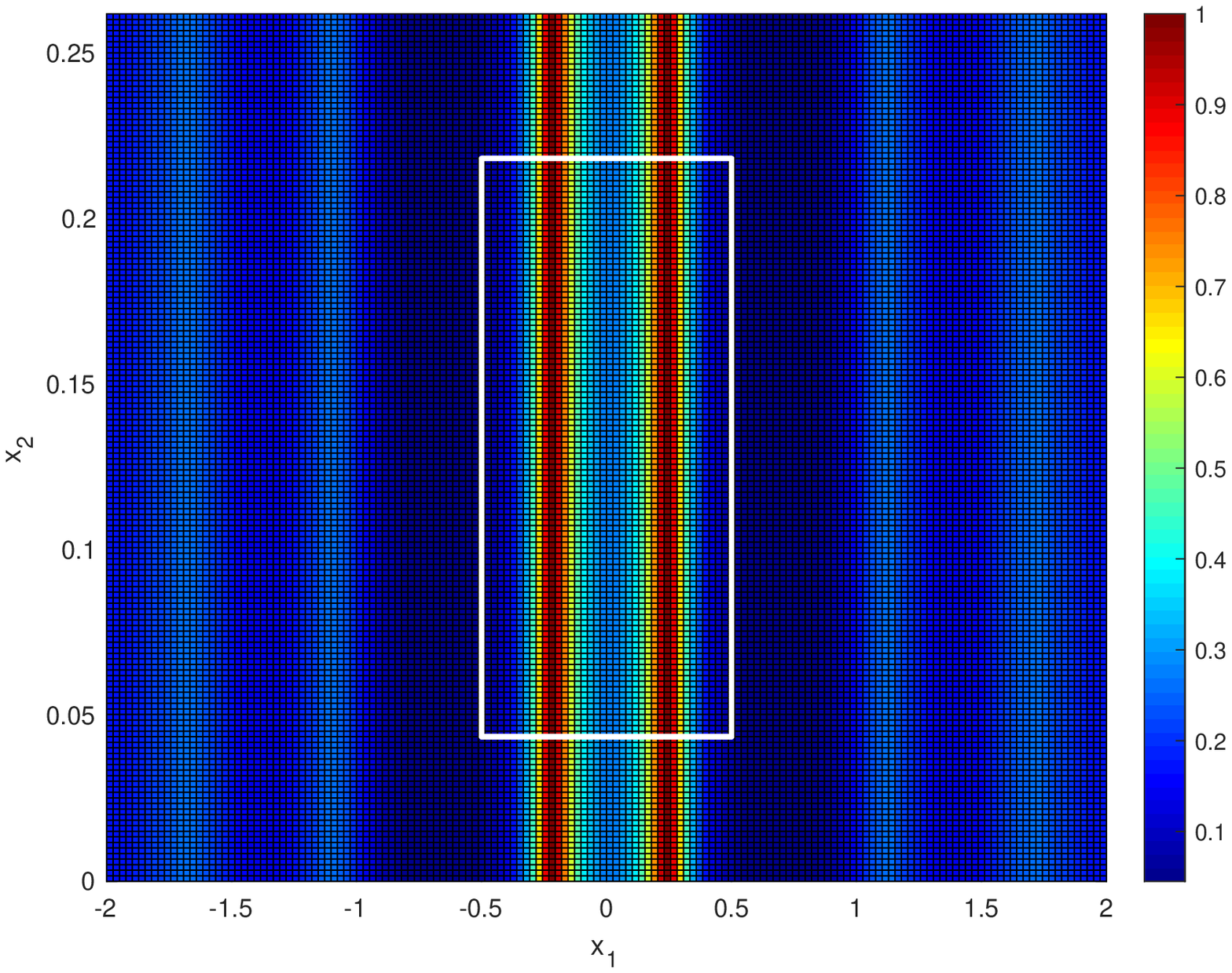}
\includegraphics[width=0.49\linewidth]{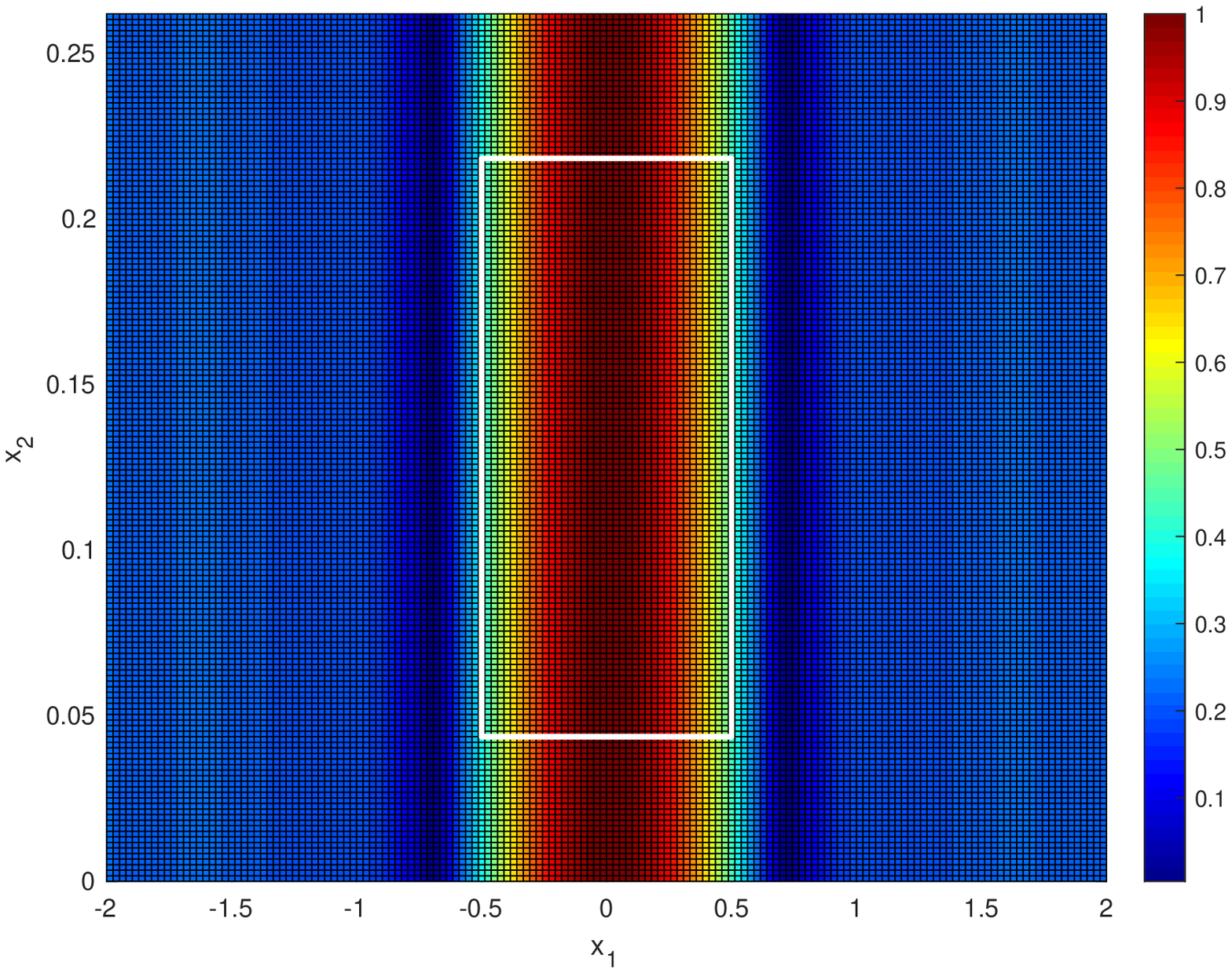}
     \caption{
     \linespread{1}
The range reconstruction of a  rectangular  source in a Neumann boundary condition waveguide. The exact support of the source is indicated by the white solid lines. Top: 47 frequencies. Middle: 23 frequencies. Bottom: 11 frequencies.  Left: the factorization method. Right: the factorization-based sampling method.
     } \label{2d Neumann fac fbsm rectangular}
    \end{figure}    
\item The factorization-based sampling method. Let $\widehat{\psi}_z$ be the   discretization of ${\psi}_{z}$ at the multiple wavenumbers (i.e. multi-frequencies)  $\{k_1,k_2,\cdots,k_N\}$. The imaging function indicated by the factorization-based sampling method is
\begin{equation*}
I_{\mbox{\tiny FBSM}}(z):=| \langle F_N \widehat{\psi}_z,\widehat{\psi}_z \rangle |.
\end{equation*}
The $\langle \cdot, \cdot \rangle$ represents the inner product of two vectors that arise from the discretization of their continuous counterparts. 
\end{enumerate}
For the visualization, we plot both $I_{\mbox{\tiny FM}}(z)$ and $I_{\mbox{\tiny FBSM}}(z)$ over the sampling region $\{-2,2\}\times \Sigma$ and we always normalize them such that their maximum value are $1$ respectively.

%In all of the following examples, we plot the range reconstruction using both the factorization method and the factorization-based sampling method. For the factorization method,  a small number $\epsilon$ needs to be chosen according to the Theorem \ref{theorem FM main}. The small number can indeed be chosen as a threshold for the resolution that is expected. We have produced the numerical examples with the limiting case  $\epsilon \to 0$ in which case  $\psi_z^\epsilon(\sigma) \to  e^{i\sigma(z_1-x_1)}$. For completeness, we also proved an example with $\epsilon=0.05$.  For the factorization-based sampling method, there is no need to choose such $\epsilon$. 

\subsection{Number of frequencies}
%     \begin{figure}[ht!]
%\includegraphics[width=0.49\linewidth]{figures/fac_cube_Neumann_23fre}
%\includegraphics[width=0.49\linewidth]{figures/fbsm_cube_Neumann_23fre}
%     \caption{
%     \linespread{1}
%The range reconstruction of a  cube  source with $23$ frequencies. The exact support of the source is indicated by the white solid lines. Left: the factorization method. Right: the factorization-based sampling method.
%     } \label{2d Neumann fac fbsm 23fre cube}
%    \end{figure}
%         \begin{figure}[ht!]
%\includegraphics[width=0.49\linewidth]{figures/fac_cube_Neumann_11fre}
%\includegraphics[width=0.49\linewidth]{figures/fbsm_cube_Neumann_11fre}
%     \caption{
%     \linespread{1}
%The range reconstruction of a  cube  source with $11$ frequencies. The exact support of the source is indicated by the white solid lines. Left: the factorization method. Right: the factorization-based sampling method.
%     } \label{2d Neumann fac fbsm 11fre cube}
%    \end{figure}

The first set of numerical examples is to illustrate how the number of frequencies affects the image. We consider a rectangular source in a two dimensional waveguide $\{-\infty,\infty\}\times (0,h)$ with Neumann boundary condition and $h=\pi/12$. In this case, the first propagating mode is $e^{ik  |x_1|}\frac{1}{\sqrt{h}}$.  The range reconstruction of the rectangular source is shown in Figure \ref{2d Neumann fac fbsm rectangular} where we have considered three different cases:
\begin{itemize}
\item Case I: 47 frequencies in the set $\{0.25, 0.5, \cdots, 11.5,11.75\}$.
\item Case II: 23 frequencies in the set $\{0.5, 1.0, \cdots, 11.0,11.5\}$.
\item Case III: 11 frequencies in the set $\{1.0, 2.0, \cdots, 10.0,11.0\}$.
\end{itemize}
It is observed that the factorization method performs better with more frequencies, while the factorization-based sampling method performs almost the same in these three different cases. { Such a difference may be due to that  the factorization method explicitly involves solving an ill-posed equation so that a more sophisticated discretization is perhaps needed in the implementation.} Another observation is that the imaging function of the factorization method sharply vanish outside the range support (which agrees with the factorization method main Theorem \ref{theorem FM main}), and the imaging function of the factorization-based sampling method gradually becomes small as the sampling point samples away from the range support (which also agrees with the main theorem of the factorization-based sampling method Theorem \ref{factorization-based SM theorem} and the property of the point spread function $S\psi_z$).

\subsection{A ``complete block'' source}
Using sampling methods to image of a complete block of the waveguide is an interesting problem \cite{monk2012sampling}. The L-shape case under consideration in this subsection is  similar to a complete block case considered in \cite{monk2012sampling}. Results in Figure \ref{2d Neumann fac fbsm Lshape}  imply that a ``complete block'' source {(i.e. a source which spans the entire cross-section)} can be reconstructed with our sampling methods with multi-frequency measurements. This is due to that the multi-frequency measurements make use of the phase (travel time) which provides information on the bulk location.

     \begin{figure}[ht!]
\includegraphics[width=0.49\linewidth]{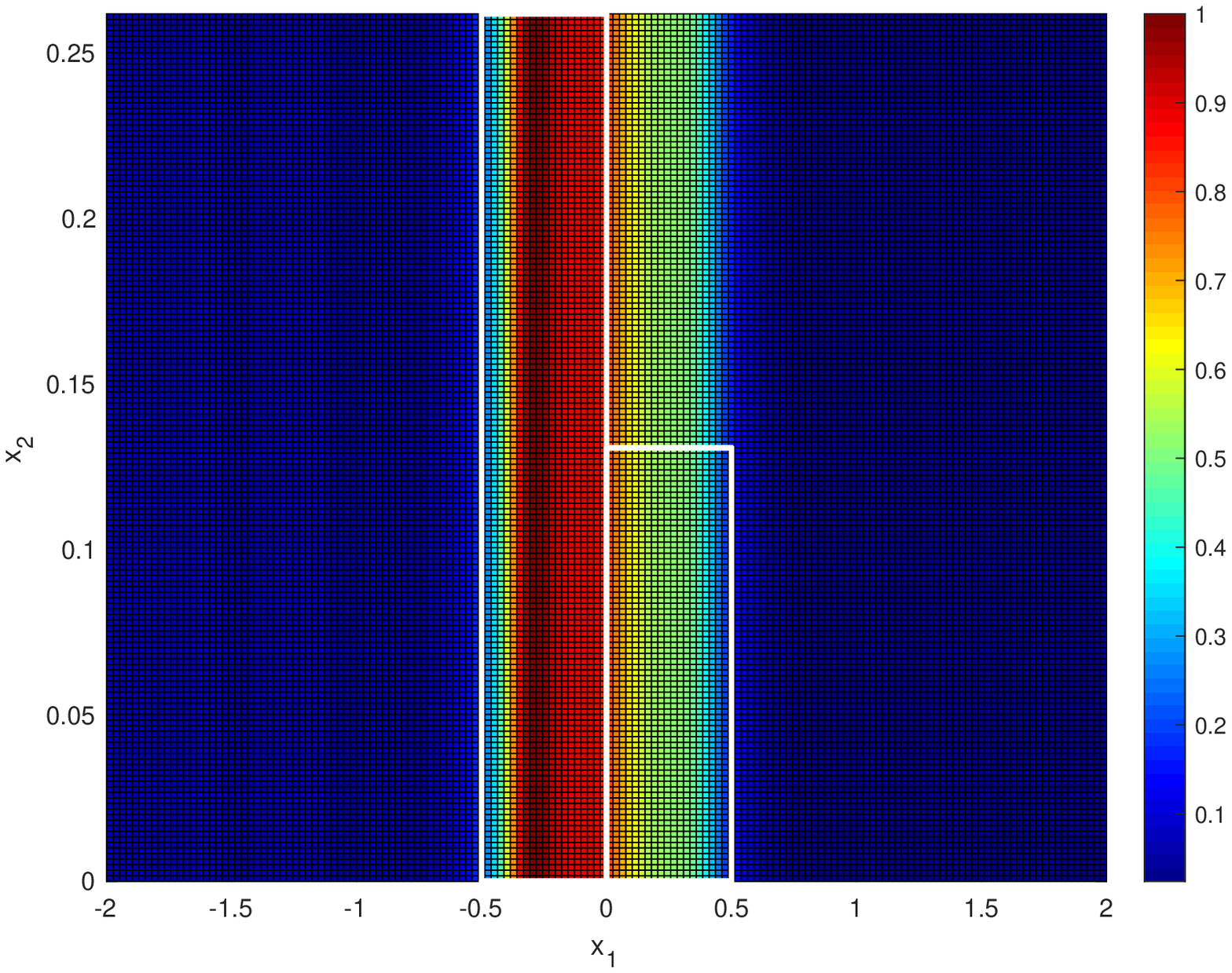}
\includegraphics[width=0.49\linewidth]{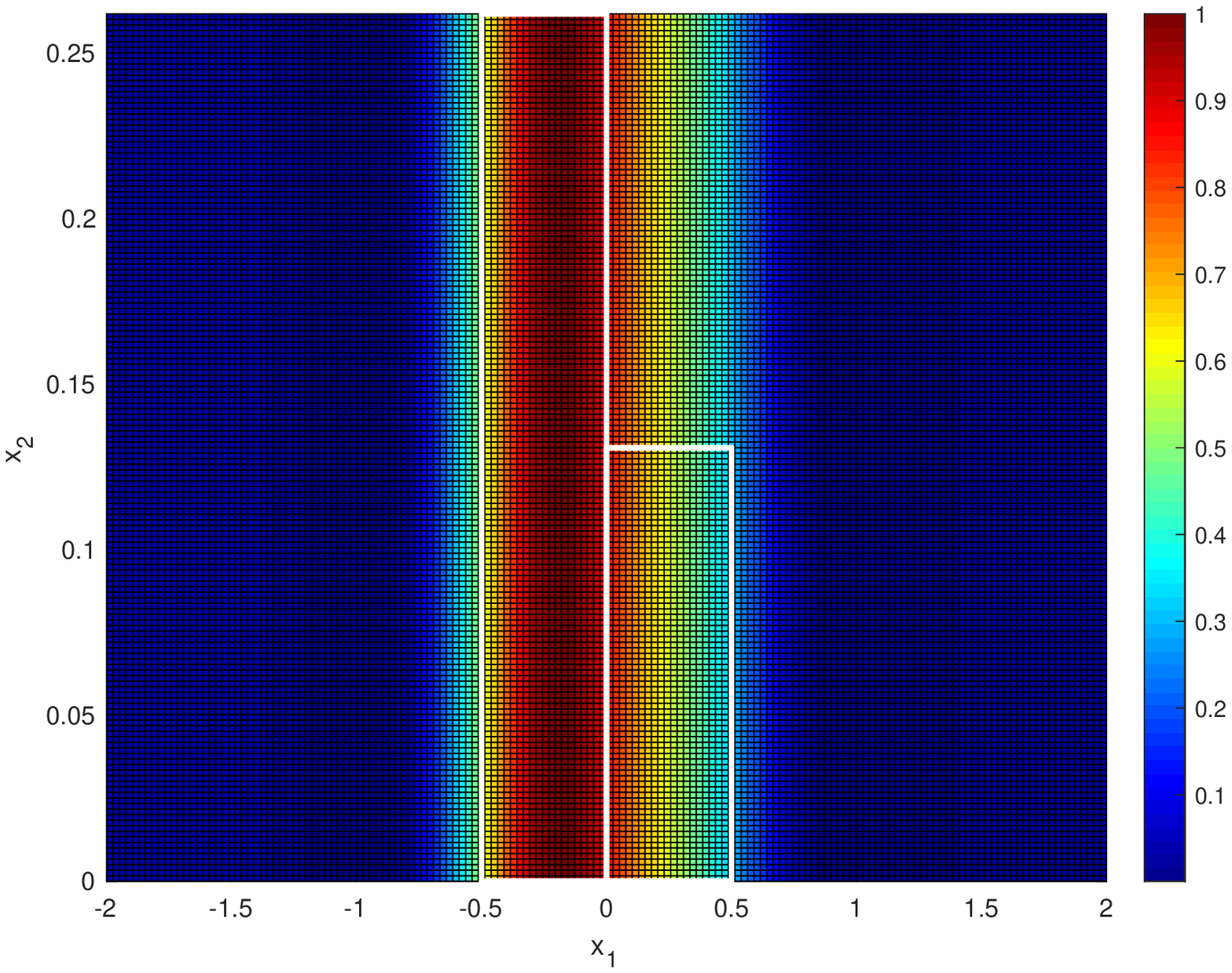}
     \caption{
     \linespread{1}
The range reconstruction of {an}  L-shape  source  in a Neumann boundary condition waveguide with 47 frequencies. The exact support of the source is indicated by the white solid lines.  Left: the factorization method. Right: the factorization-based sampling method.
     } \label{2d Neumann fac fbsm Lshape}
    \end{figure}
%    
%         \begin{figure}[ht!]
%\includegraphics[width=0.49\linewidth]{figures/fac_Lshape_Neumann_47fre}
%\includegraphics[width=0.49\linewidth]{figures/fbsm_Lshape_Neumann_47fre}
%     \caption{
%     \linespread{1}
%The range reconstruction of a  L-shape  source with $47$ frequencies. The exact support of the source is indicated by the white solid lines. Left: the factorization method. Right: the factorization-based sampling method.
%     } \label{2d Neumann fac fbsm 47fre Lshape}
%    \end{figure}
%    

\subsection{Different boundary conditions}
We further illustrate the performance of the sampling methods with different boundary conditions of the waveguide. In this regard, a two dimensional waveguide $\{-\infty,\infty\}\times (0,h)$ with a mixed boundary condition and $h=\pi/12$ is considered. The mixed boundary condition  is a Dirichlet boundary condition on the top boundary and a Neumann boundary condition on the bottom boundary. 
The first propagating mode in this case is $e^{i\sqrt{k^2-(\pi/2h)^2}  |x_1|}\sqrt{\frac{2}{h}} \cos(\pi x_\perp /2h)$, and the corresponding dispersion relation reads $
\mu_1(k) = \sqrt{k^2-(\pi/2h)^2}$. Note that this mixed boundary condition yields a non-linear dispersion relation. The first passband corresponds to $k \in (6,12)$, and the following numerical example is produced with 41  wavenumbers (i.e.  frequencies) in the set $\{k=\sqrt{(\pi/2h)^2+\sigma^2}: \sigma = 0.25, 0.5, \cdots, 10.0,10.25\}$.

The first row of Figure \ref{2d mixed fac fbsm rectangular rhombus} is the range reconstruction of a rectangular source. The support of the source is the same as the one in Figure \ref{2d Neumann fac fbsm rectangular}  in order to compare performance with respect to different boundary conditions, and it is observed that the performances of the sampling methods in these two different boundary condition cases are similar. The difficulty arising from the non-linear dispersion relation in the mixed boundary condition case  is overcame via the carefully defined far-field operator.

The second row  of Figure \ref{2d mixed fac fbsm rectangular rhombus} is the range reconstruction of a rhombus shape source. Note that the rhombus shape source does not satisfy the assumption in Corollary \ref{corollary FM main}. In particular the left and right corner of the rhombus shape source would be reconstructed with certain tolerance for the factorization method  as indicated by Theorem \ref{theorem FM main} and Lemma \ref{lemma range S*} (and its proof). Since the left and right corners are sharp, the factorization-based sampling method may also reconstruct the support with certain tolerance. These observations are illustrated by the the second row  of Figure \ref{2d mixed fac fbsm rectangular rhombus}.

     \begin{figure}[ht!]
     \includegraphics[width=0.49\linewidth]{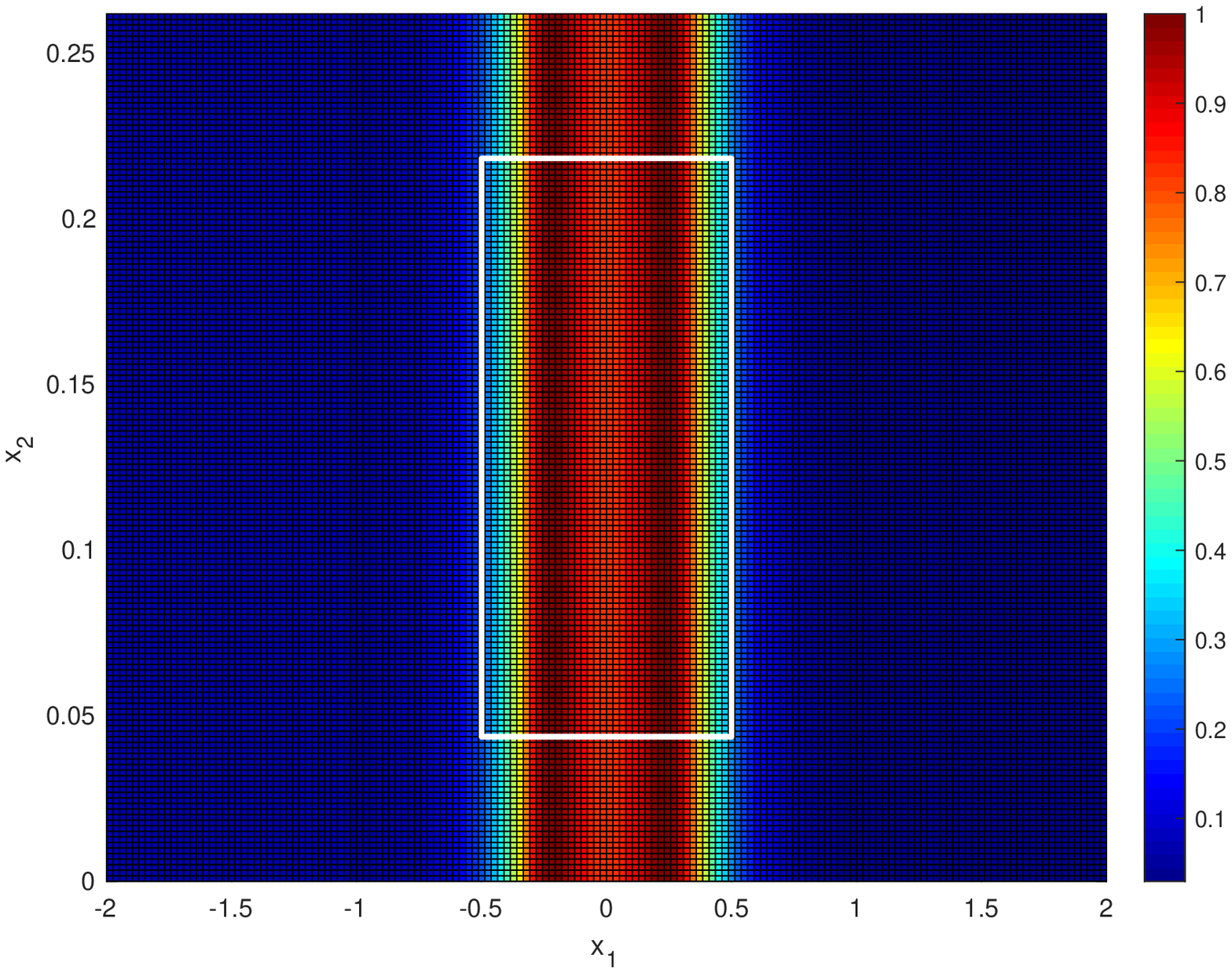}
\includegraphics[width=0.49\linewidth]{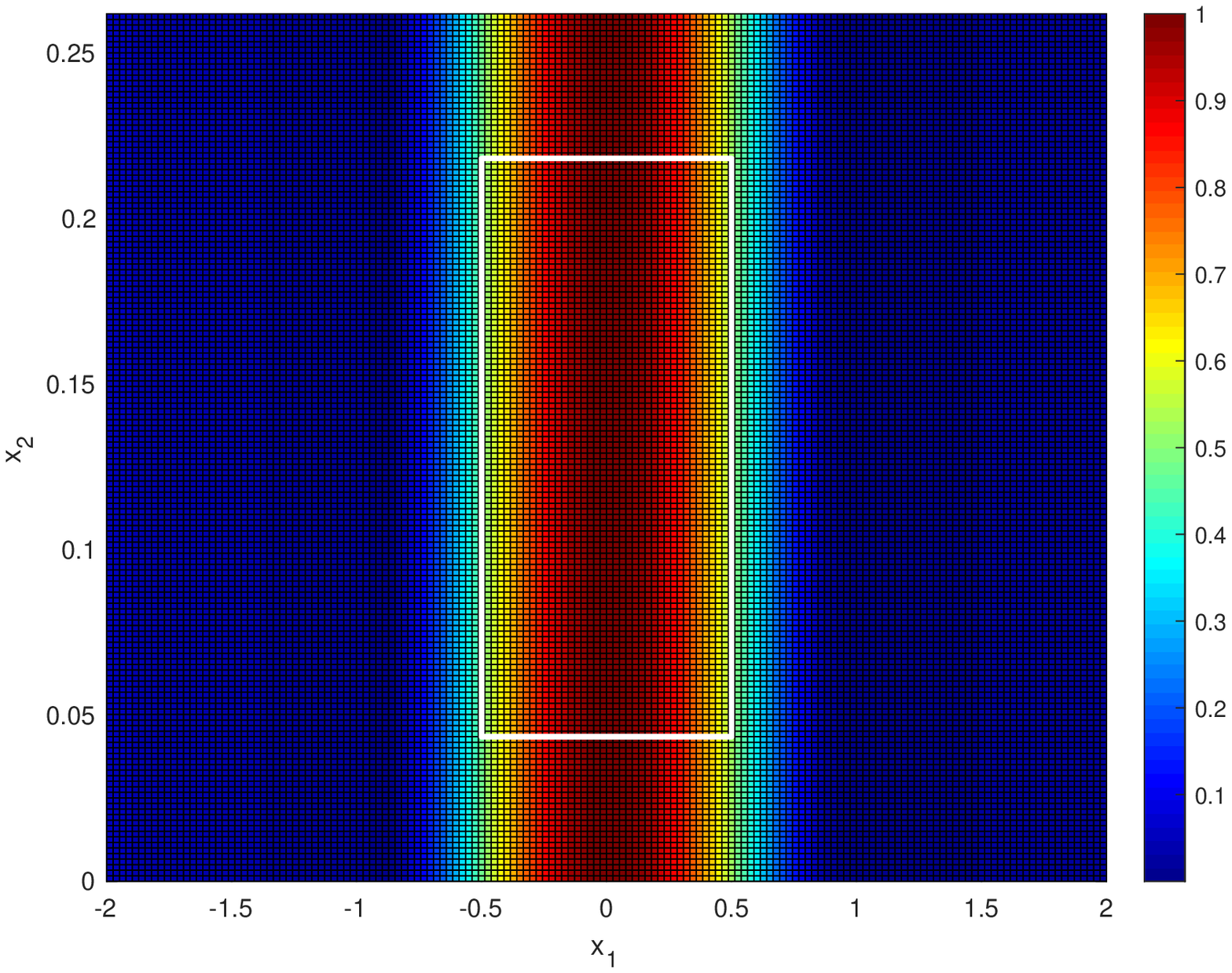}
\includegraphics[width=0.49\linewidth]{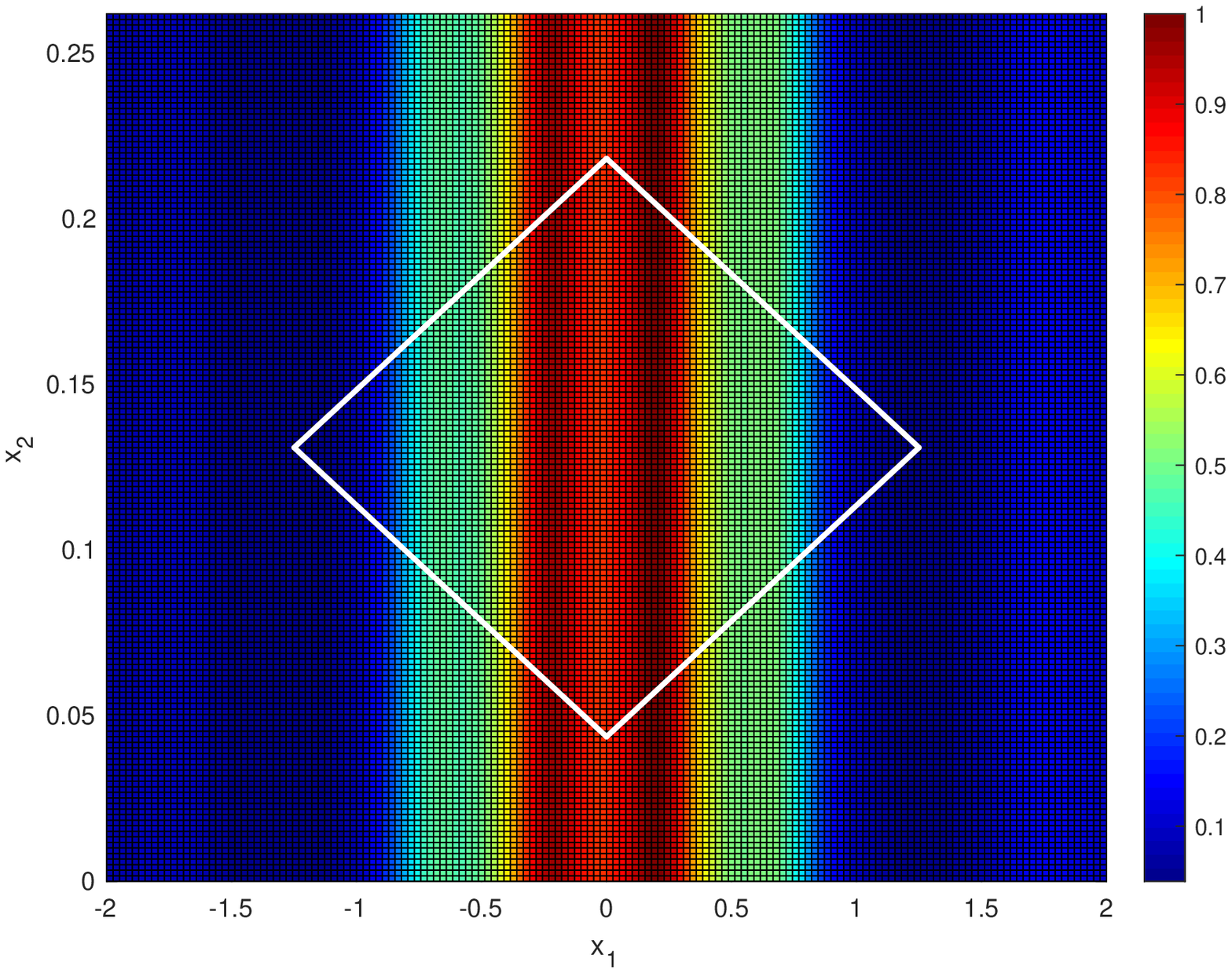}
\includegraphics[width=0.49\linewidth]{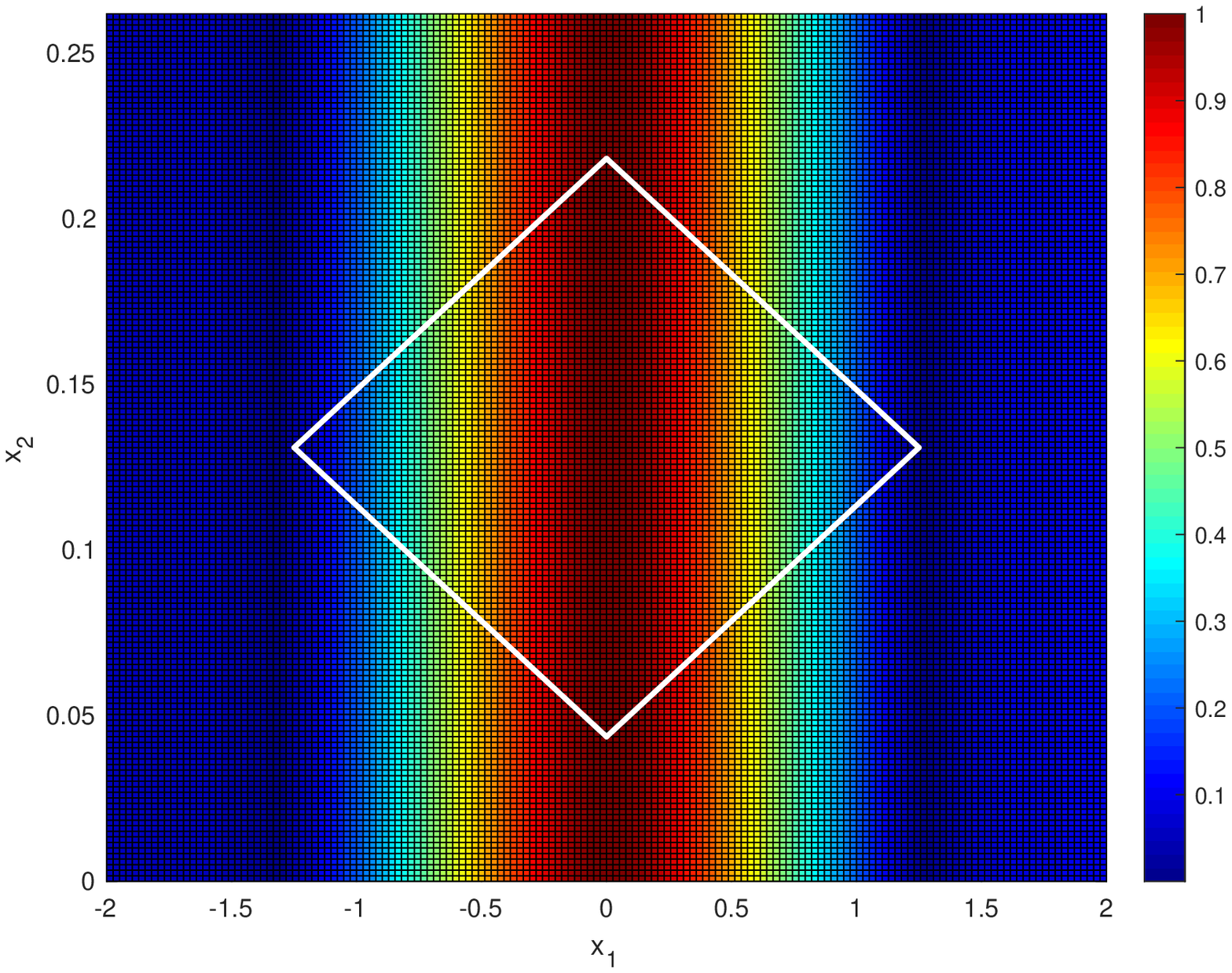}
     \caption{
     \linespread{1}
The range reconstruction of sources in a mixed boundary condition waveguide with 41 frequencies. The exact support of the source is indicated by the white solid lines. Top: a rectangular source. Bottom: a rhombus shape source. Left: the factorization method. Right: the factorization-based sampling method.
     } \label{2d mixed fac fbsm rectangular rhombus}
    \end{figure}
    
%\subsection{Three dimensional case}
%
%Finally we include a three dimensional case by considering a waveguide $\{-\infty,\infty\}\times \Sigma$ with $\Sigma = (0,a) \times (0,b)$ and Neumann boundary condition, where $a=\pi, b=\pi$. The range reconstruction in \ref{} characterize the range support of the source with good fidelity.

\subsection{A complete sound-soft block}  

In the last subsection, we illustrate how to image a complete block (with sound soft boundary condition) by a preliminary numerical example. It is observed in \cite{monk2012sampling} that, a complete block of a waveguide can be reconstructed with near field measurements, however it may not be reconstructed with far-field measurements using the linear sampling method at a single frequency. With the multi-frequency measurements in our framework, it is possible to image such a complete block using far-field measurements. Though we have demonstrated this via the inverse source problem, to be more convincing we show a preliminary example for imaging a sound soft block. The waveguide is $(-\infty,\infty)\times (0,h)$ with Neumann boundary condition and $h=\pi/12$. The complete block is at $-0.5 \times (0,h)$. In this case, the scattered wave field at far-field $x_r \in \{-10\}\times (0,2)$ due to a point source at transmitter location $x_s \in \{-10\}\times (0,2)$ can be written explicitly and we further add $5\%$ noise to this wave field to generate our synthetic data. 

The difference between the block case and the source case can be interpreted though the travel time: the travel time in the block case roughly doubles the one in the source case (where the block and the source are at the same location); this is because the wave travel from the transmitter  to the block then back to the receiver in the block case while the wave just travel from the source to the receiver in the source case. Therefore the functions $\psi_z^\epsilon$ and $\psi_z$ have to be modified by doubling their phases, i.e. $\psi_z^\epsilon(\sigma) \to \frac{1}{|B(z,\epsilon)|}\int_{B(z,\epsilon)}e^{i2\sigma(y_1-x_1^*)}  \ind y$ and $\psi_z(\sigma) \to  e^{i 2 \sigma(z_1-x_1^*)} \sqrt{\psi_1(z_\perp)}$. This is the only change in our implementation to image a complete block with the sampling methods. Figure \ref{2d neumann fac fbsm block} illustrate the performance of the sampling methods to image a complete sound soft block. $47$ frequencies in the set $\{0.25, 0.5, \cdots, 11.5,11.75\}$ are used. It is observed that the block is correctly located and a  sharper result is given by the factorization method. The obstacle/block case is more complicated than the source case and a complete theoretical justification is part of our future work.

     \begin{figure}[ht!]
     \includegraphics[width=0.49\linewidth]{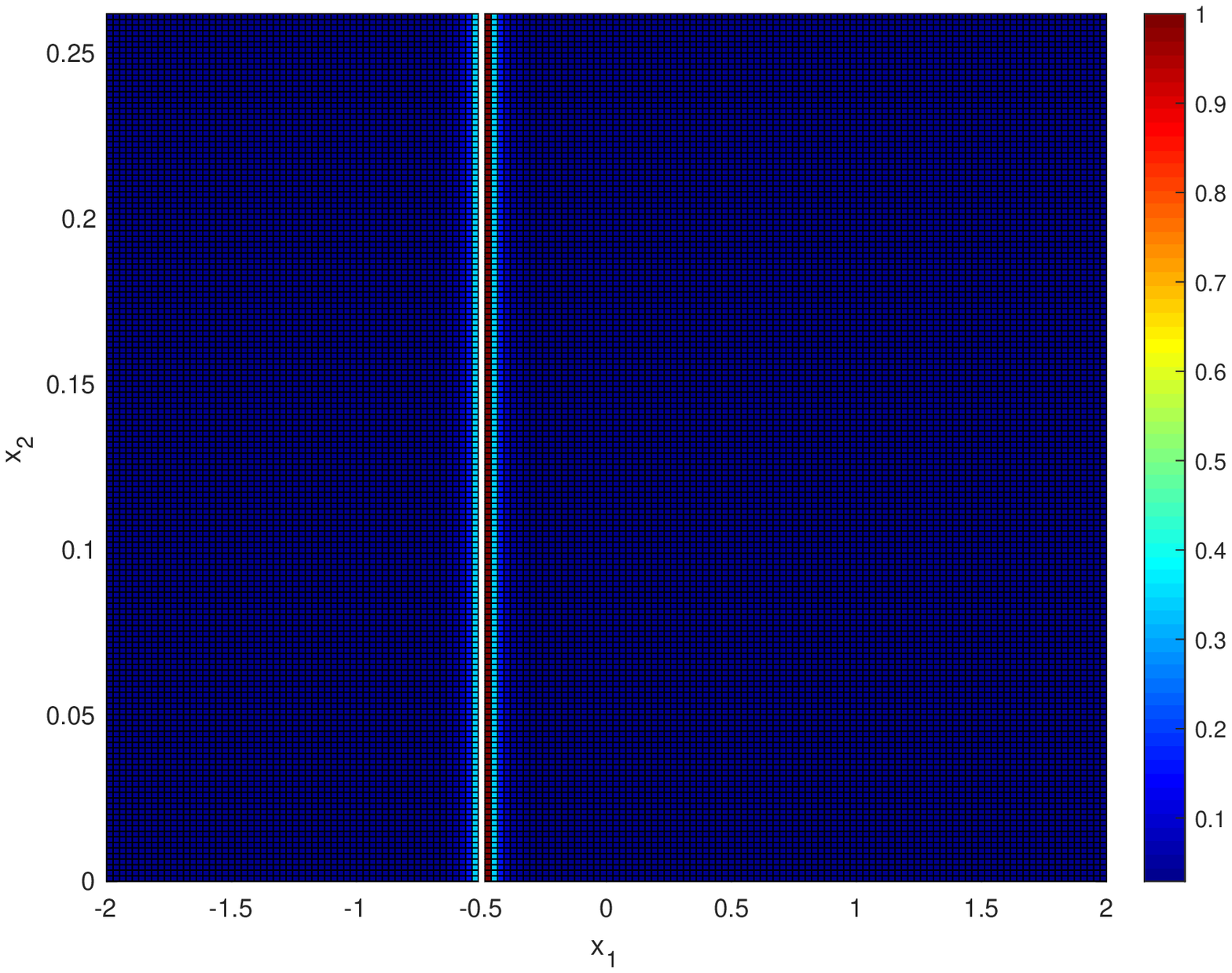}
\includegraphics[width=0.49\linewidth]{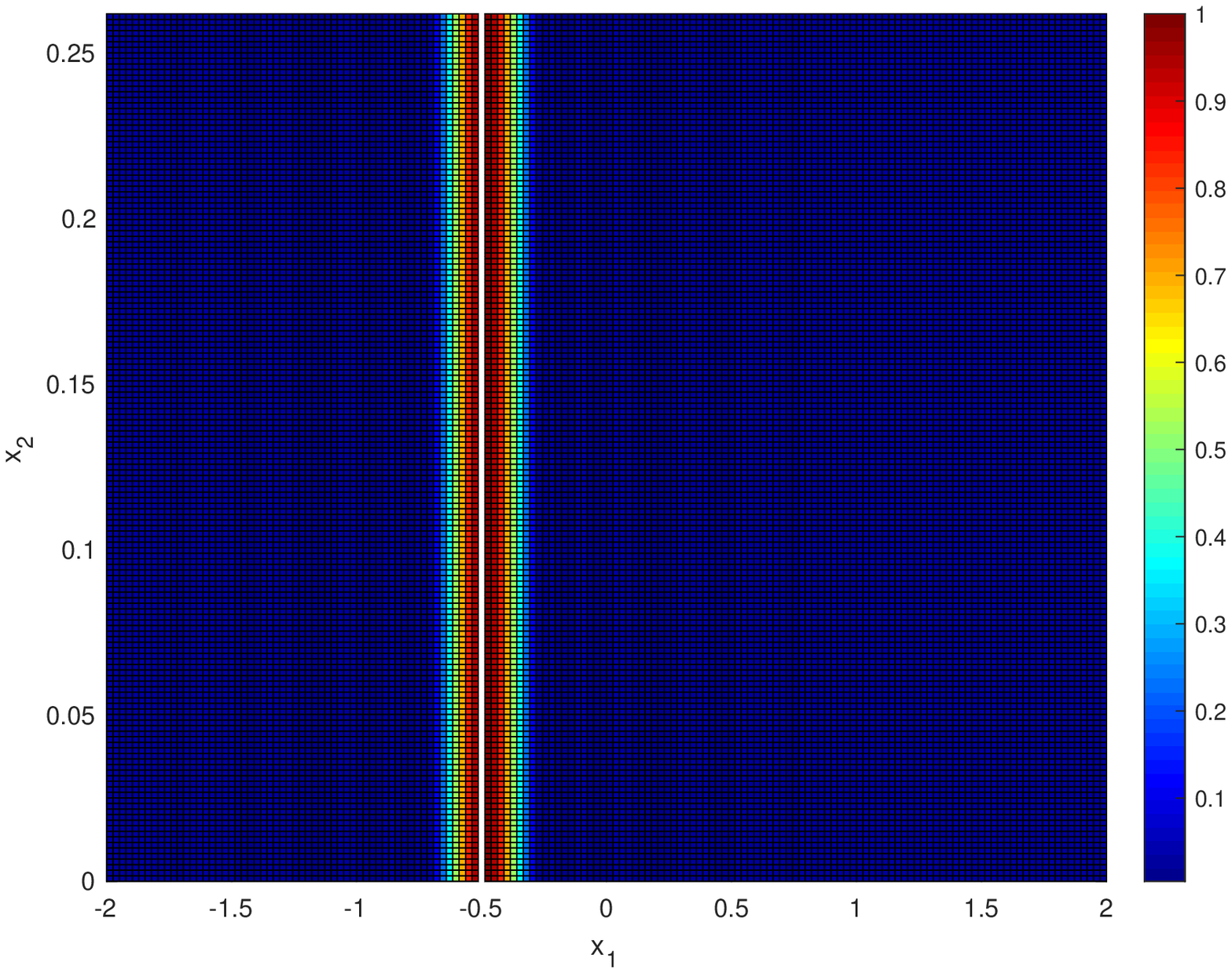}
     \caption{
     \linespread{1}
The range reconstruction of a complete {(Dirichlet)} block in a Neumann boundary condition waveguide with 47 frequencies. The exact support of the source is indicated by the white solid lines. Left: the factorization method. Right: the factorization-based sampling method.
     } \label{2d neumann fac fbsm block}
    \end{figure}

%%%%%%%%%%%%%%%%%%%%%%%%%%%%%%%%%%%%%%%%%%
\section{Remark on designs of multi-frequency operator and assumptions on the source} \label{section design multi-frequency operator and source assumption}
In the above discussions of the multi-frequency sampling methods, we have made the assumption on the source $f$ that $e^{i\theta}f>c_1>0$ or $e^{i\theta}f<c_2<0$ with some $\theta \in [0,2\pi)$. Indeed it is possible to relax this assumption by designing other multi-frequency operators. We introduce in this section another two multi-frequency operators to remark that the sampling methods can require less restrictive conditions on the source; this may be of interest to readers who is more interested in minimal assumption on the source.
%For a thorough discussion on designs of multi-frequency operator to relax assumptions on the source, we introduce in this section another two multi-frequency operators so that the sampling methods may require less restrictive conditions (compared to the condition in Section \ref{section operator}) on the source. 
The first operator in Section \ref{multi-frequency operator backward and forward measurements} is to consider both ``backward and forward scattering'' measurements; the second operator in Section \ref{section multi-frequency operator alpha} still uses only the ``backscattering'' measurements (which is less obvious to design) and it usually uses information in part of the frequency range. Their advantages are that the sampling methods with these operators would require a less restrictive assumption that $\Re(f(x) e^{i\tau} ) > c_3>0 $ almost everywhere for $x\in D$ with some constant phase $\tau$; their drawbacks are that either more measurements have to be used (Section \ref{multi-frequency operator backward and forward measurements}) or information in part of the frequency range can usually be utilized (Section \ref{section multi-frequency operator alpha}). The first technique in Section \ref{multi-frequency operator backward and forward measurements} is in the same spirit of the inverse source problem in the whole space  \cite{GriesmaierSchmiedecke-source} (that uses measurements from two opposite directions); the second technique in Section \ref{section multi-frequency operator alpha} seems new which may facilitate further advancement of the multi-frequency factorization method using only the backscattering measurements. In the following we briefly show how to generalize the results in Section \ref{section operator}--\ref{section sampling methods} using these two multi-frequency operators. 
 
\subsection{A multi-frequency operator based on ``backward and forward scattering'' measurements} \label{multi-frequency operator backward and forward measurements}
%The assumption on $f$ that $e^{i\theta}f>c_1>0$ or $e^{i\theta}f<c_2<0$ with some $\theta \in [0,2\pi)$ can indeed be removed for Theorem \ref{theorem factorization S*TS} by considering 
One way to relax the assumption on the source is to consider both the ``backward and forward scattering'' measurements, i.e. measurements at {$x_l:=(x^*_1,x^*_\perp)$ and $x_r:=(-x^*_1,x^*_\perp)$}. In this case, we introduce a similar
multi-frequency far-field operator $\widetilde{F}:L^2(k_-,k_+) \to L^2(k_-,k_+)$ by
{\small
\begin{equation} \label{definition far-field operator modify}
(\widetilde{F} g)(\sigma;x_l,x_r) :=  \int_{k_-}^{k_+} { -i } \mu_1(\omega_{\sigma\gamma}) \Big[H(\sigma-\gamma)u_p^s\left({x_l}; \omega_{\sigma\gamma}\right){+} H(\gamma-\sigma) e^{2i {|\sigma-\gamma|}x^*_1}  {u_p^s\left({x_r}; \omega_{\sigma\gamma}\right)} \Big] g(\gamma) \ind \gamma,
\end{equation}
}
$\sigma \in (k_-,k_+)$. 
%In this definition, we have used measurements at $\{\pm a\} \times \Sigma$. 
One can then verify that
{\small
\begin{eqnarray*}
 (\widetilde{F} g)(\sigma;x_l,x_r)     
% \int_{k_-}^{k_+}  { -i } \mu_1(\omega_{\sigma\gamma}) \Big[H(\sigma-\gamma)u_p^s\left({x_l}; \omega_{\sigma\gamma}\right) {+} H(\gamma-\sigma) e^{2i|\sigma-\gamma| { x_1^*}} {u_p^s\left({x_r}; \omega_{\sigma\gamma}\right)} \Big] g(\gamma) \ind \gamma \\
&=&1/2\int_{k_-}^{k_+}  \int_D  \psi_1(y_\perp)e^{i (\sigma-\gamma) |x^*_1-y_1|}  f(y) { \psi_1(x^*_\perp) }g(\gamma)  \ind y \ind \gamma \\
&=&  \int_D e^{i \sigma |x^*_1-y_1|} \overline{\sqrt{\psi_1(y_\perp)} }\left[\left(\int_{k_-}^{k_+}  \sqrt{\psi_1(y_\perp)}e^{-i \gamma |x^*_1-y_1|} g(\gamma)  \ind \gamma \right) \frac{  f(y) { \psi_1(x^*_\perp) } }{2}  \right] \ind y,
\end{eqnarray*}
}
and we can proceed as in Theorem \ref{theorem factorization S*TS} to conclude that $
\widetilde{F}=S^*\widetilde{T}S
$ with 
%$\widetilde{T}:L^2(D)\to L^2(D)$ where $\widetilde{T} h = \frac{ f{ \psi_1(x^*_\perp) }h}{2}$. 
$\widetilde{T}: L^2(D) \to L^2(D)$ 
\begin{equation} \label{operator T0}
(\widetilde{T} h)(y) :=  \frac{\psi_1(x^*_\perp)f(y)h(y)  }{2}   , \quad y \in D.
\end{equation}

Though there is no assumption on $f$ to derive the factorization $\widetilde{F}=S^*\widetilde{T}S$  in the ``backward and forward scattering'' measurement case, an assumption that $\Re(f(x) e^{i\tau} ) >{c_3>0}$ almost everywhere for $x\in D$ with some $\tau$ would then be added to ensure the coercivity of an operator related to $\widetilde{T}$. 
%We include relevant details in the next subsection (as $\widetilde{T}$ also appear there). 
To show the coercivity, we introduce for a generic bounded operator $A: L^2 \to L^2$ that
$$
\Re A = \frac{A+A^*}{2}, \quad \Im A = \frac{A-A^*}{2i}.
$$
Now we have
\begin{lemma} \label{lemma T0 coercive}
Suppose that $\Re ( f(x) e^{i\tau}) \ge c_3>0, \forall x\in D$ for some constant phase $\tau$, then we have that $\Re (e^{i\tau}\widetilde{T})$ is self-adjoint and coercive.
\end{lemma}
\begin{proof}
From the assumption that $\Re (f e^{i\tau}) \ge c_3>0$ for some constant phase $\tau$, we have that 
\begin{equation}\label{operator re tilde T  e i tau}
\Big(\Re (e^{i\tau}\widetilde{T}) h \Big)(y) :=  \frac{\psi_1(x^*_\perp)\Re (f e^{i\tau})h(y)  }{2}   , \quad y \in D,
\end{equation}
and that
\begin{eqnarray}
 \langle \Re (e^{i\tau}\widetilde{T}) h,h\rangle_{L^2(D)} 
&=& \int_D \frac{ \Re (e^{i\tau}f(y)) {  \psi_1(x^*_\perp) }|h(y)|^2  }{2}  \ind y   \ge  c'\|h\|^2_{L^2(D)} 
 \label{lemma T0 coercive proof eqn1}
\end{eqnarray}
for some positive constant $c'$. This completes the proof.
\end{proof}
Lemma \ref{lemma T0 coercive} would then allow us to derive a $F\#$ version of the factorization method (that  requires less restrictive conditions on the source $f$) which is in the same spirit of the inverse source problem in the whole space  \cite{GriesmaierSchmiedecke-source} (that uses measurements from two opposite directions). However  we are mainly concerned with the most interesting ``backscattering'' case and we shall  discuss another possible  multi-frequency operator $F_\alpha$ in the next section. We then finalize by generalizing the main results of the factorization method and the factorization-based sampling method using $\widetilde{F}$ and $F_\alpha$.
%%%%%%
{
\subsection{Another multi-frequency operator based on ``backscattering'' measurements} \label{section multi-frequency operator alpha}
Indeed we can still impose less restrictive conditions on the source $f$ in the case of ``backscattering'' by designing another multi-frequency operator as follows. Here we introduce a multi-frequency far-field operator $F_\alpha:L^2(k_-,k_+(\alpha)) \to L^2(k_-,k_+(\alpha))$ by
\begin{equation} \label{definition far-field operator alpha}
(F_\alpha g)(\sigma;x^*) :=  \int_{k_-}^{k_+(\alpha)} -i\mu_1(\omega_{\sigma\gamma \alpha})  u_p^s\left(x^*; \omega_{\sigma\gamma \alpha}\right) g(\gamma) \ind \gamma, \quad \sigma \in (k_-,k_+(\alpha))
\end{equation}
where $\omega_{\sigma\gamma \alpha}:= \sqrt{\lambda_1^2+ \big(\sigma-\gamma + \frac{1}{\alpha} \sqrt{\lambda_2^2-\lambda_1^2}\big)^2}$ with $\alpha\ge 2$ being a  positive constant that is determined later (by the coercivity assumption of a certain operator),  $k_- =0$ and $k_+(\alpha):= \frac{1}{\alpha}\sqrt{\lambda_2^2-\lambda_1^2}$. Note that for $\sigma, \gamma \in (k_-,k_+(\alpha))$, $\omega_{\sigma \gamma \alpha} \in \Big(\lambda_1,\sqrt{\lambda^2_1 + \frac{4}{\alpha^2} (\lambda^2_2-\lambda^2_1)}\Big) \subseteq (\lambda_1,\lambda_2)$.

% { Note that $\omega_{\sigma\gamma \alpha}$  allows us to define the following multi-frequency far-field operator which requires less restrictive conditions on $f$.}
%We introduce the single mode multi-frequency far-field operator $F_\alpha:L^2(k_-,k_+(\alpha)) \to L^2(k_-,k_+(\alpha))$ by
%{
%\begin{equation} \label{definition far-field operator alpha}
%(F_\alpha g)(\sigma;x^*) :=  \int_{k_-}^{k_+(\alpha)} -i\mu_1(\omega_{\sigma\gamma \alpha})  u_p^s\left(x^*; \omega_{\sigma\gamma \alpha}\right) g(\gamma) \ind \gamma, \quad \sigma \in (k_-,k_+(\alpha))
%\end{equation}
%} 
%where $k_- =0$ and $k_+(\alpha)= \frac{1}{\alpha}\sqrt{\lambda_2^2-\lambda_1^2}$ so that $\omega_{\sigma \gamma} \in (\lambda_1,\lambda_2)$. 
%

 In this case, similar to the role of $\omega_{\sigma \gamma}$ in Section \ref{section operator}, one finds via a direct  calculation that $\mu_1(\omega_{\sigma\gamma\alpha}) = \sigma-\gamma+\frac{1}{\alpha}\sqrt{\lambda_2^2-\lambda_1^2}$ (for $\sigma, \gamma \in (k_-,k_+(\alpha))$) which linearizes the dispersion in another (though less obvious) way. In principle, one may find other ways to design multi-frequency operators.
The possible drawback is seen as that $\omega_{\sigma \gamma \alpha}$ has range in $\Big(\lambda_1,\sqrt{\lambda^2_1 + \frac{4}{\alpha^2} (\lambda^2_2-\lambda^2_1)}\Big)$ which is a subset of $(\lambda_1,\lambda_2)$ if $\alpha>2$, this implies that  information in part of the frequency range is only used if $\alpha>2$. Nevertheless, with such a design of the multi-frequency operator, we are able to derive a factorization method with  only the ``backscattering'' measurements with  less restrictive conditions on the source $f$. We also refer to \cite{GriesmaierSchmiedecke-source} for a discussion on  an extension of the factorization method to ``backscattering'' measurements.
%Such techniques may advance the study on the multi-frequency factorization method in waveguide or whole space using only the backscattering measurements.
%We hope that the techniques developed in waveguide may in return facilitate the study on the multi-frequency factorization method in the whole space.
%In the above definition of $F$, the integral is understood in the Lesbegue sense. 

The operators $S_\alpha: L^2(k_-,k_+(\alpha)) \to L^2(D)$ and $S^*_\alpha: L^2(D) \to L^2(k_-,k_+(\alpha))$ remain formally the same as in \eqref{operator S} and \eqref{operator S*}, respectively (by only replacing $k_+$ by $k_+(\alpha)$). The difference is the operator $T_\alpha: L^2(D) \to L^2(D)$ defined via
\begin{equation} \label{operator T alpha}
(T_\alpha h)(y) :=  \frac{e^{i \frac{1}{\alpha}\sqrt{\lambda_2^2-\lambda_1^2}|x_1^*-y_1|} \psi_1(x^*_\perp)f(y)h(y)  }{2}   , \quad y \in D.
\end{equation}
Using the definition of $F_\alpha$ \eqref{definition far-field operator alpha}, we can prove similarly that
{\footnotesize
\begin{eqnarray*}
&&(F_\alpha g)(\sigma;x^*) =   
%\int_{k_-}^{k_+(\alpha)}  {  -i } \mu_1(\omega_{\sigma\gamma \alpha})  u_p^s\left(x^*; \omega_{\sigma\gamma \alpha}\right)  g(\gamma) \ind \gamma \\
 1/2\int_{k_-}^{k_+(\alpha)} \int_D  \psi_1(y_\perp) { e^{i \left(\sigma-\gamma+\frac{1}{\alpha} \sqrt{\lambda_2^2-\lambda_1^2}\right) |x^*_1-y_1|}     \psi_1(x^*_\perp)}  f(y) g(\gamma)  \ind y \ind \gamma \\
&=&   
\int_D e^{i \sigma |x^*_1-y_1|} \overline{\sqrt{\psi_1(y_\perp)} }\Bigg[\left(\int_{k_-}^{k_+(\alpha)} \sqrt{\psi_1(y_\perp)}e^{-i \gamma |x^*_1-y_1|} g(\gamma)  \ind \gamma \right) \frac{{ e^{i \frac{1}{\alpha} \sqrt{\lambda_2^2-\lambda_1^2} |x^*_1-y_1| } \psi_1(x^*_\perp)} f(y)  }{2}  \Bigg] \ind y \\
&=&  \big( S_\alpha^*T_\alpha S_\alpha g \big) (\sigma).
\end{eqnarray*}
}
Next we investigate the coercivity of the operator $\Re(e^{i\tau} T_\alpha)$.
Recall that $\widetilde{T}: L^2(D) \to L^2(D)$ is given by \eqref{operator T0} in Section \ref{multi-frequency operator backward and forward measurements}.
%\begin{equation} \label{operator T0}
%(\widetilde{T} h)(y) :=  \frac{\psi_1(x^*_\perp)f(y)h(y)  }{2}   , \quad y \in D.
%\end{equation}
\begin{theorem} \label{theorem T alpha coercivity}
Suppose that $\Re(e^{i\tau}\widetilde{T})$ is positive and coercive for some constant phase $\tau$. Let $T_\alpha$ be given by  \eqref{operator T alpha}, then there exists a positive constant $\alpha \ge 2$ such that the operator $\Re(e^{i\tau} T_\alpha): L^2(D) \to L^2(D)$ is self-adjoint, coercive  and  
$$
\langle \Re(e^{i\tau} T_\alpha) h,h\rangle_{L^2(D)} \ge c\|h\|^2_{L^2(D)},
$$
for some positive constant $c$.
\end{theorem}
\begin{proof}
We first introduce $\zeta :=\frac{1}{\alpha}\sqrt{\lambda_2^2-\lambda_1^2}|x_1^*-y_1|$ in the proof  to avoid writing long equations so that the presentation  is clean and compact.
From the definitions of $T_\alpha$ in \eqref{operator T alpha}, we have that
{\small
\begin{equation*} \label{operator T alpha e i tau }
\Big( (e^{i\tau} T_\alpha) h\Big)(y) =  \frac{e^{i  \zeta} \psi_1(x^*_\perp) e^{i\tau} f(y)h(y)  }{2}, \quad
%\end{equation*}
%and
%\begin{equation*} \label{operator T alpha e i tau conjugate}
\Big( (e^{i\tau} T_\alpha)^* h\Big)(y)  =  \frac{e^{-i \zeta} \psi_1(x^*_\perp) \overline{e^{i\tau} f(y)}h(y)  }{2}   , \quad y \in D,
\end{equation*}
}
whereby 
\begin{eqnarray} \label{operator re T alpha e i tau}
&&\hspace{-1cm}\Big( \Re(e^{i\tau} T_\alpha)  h\Big)(y)  =  \Big( \cos ( \zeta  ) \Re(e^{i\tau} f(y)) -  \sin( \zeta) \Im(e^{i\tau} f(y)) \Big) 
%  \nonumber \\
%&&\hspace{3.5cm} \times 
\frac{ \psi_1(x^*_\perp)  h(y)  }{2}, \quad y\in D.
\end{eqnarray}
Now from the expressions of $\Re(e^{i\tau} T_\alpha) $ in \eqref{operator re T alpha e i tau} and of $\Re(e^{i\tau}\widetilde{T})$ in \eqref{operator re tilde T  e i tau},
\begin{eqnarray}
%&&\left|\big|\langle \Re(e^{i\tau} T_\alpha) h,h\rangle_{L^2(D)} \big|  - \big|\langle \Re(e^{i\tau}\widetilde{T}) h,h\rangle_{L^2(D)} \big| \right| \nonumber\\
&& \left| \langle \Re(e^{i\tau} T_\alpha)h,h\rangle_{L^2(D)} - \langle \Re(e^{i\tau}\widetilde{T}) h,h\rangle_{L^2(D)} \right| \nonumber \\
&=& \Big|\int_D  \Big( \cos(  \zeta ) \Re(e^{i\tau} f(y)) -\Re(e^{i\tau} f(y)) -  \sin(\zeta) \Im(e^{i\tau} f(y)) \Big)  
%\nonumber\\
%&&\qquad \times  
\frac{ \psi_1(x^*_\perp)  |h(y)|^2  }{2}  \ind y \Big| . \qquad\label{theorem T alpha coercivity proof eqn1}
\end{eqnarray}
Having $y\in D$ in a fixed sampling region, then $|\sqrt{\lambda_2^2-\lambda_1^2}(y_1-x_1^*)|$ is bounded above so that we can always choose $\alpha$ such that $|\zeta|=|\frac{1}{\alpha} \sqrt{\lambda_2^2-\lambda_1^2}(y_1-x_1^*)| < \epsilon'$ for a sufficiently small $\epsilon'>0$, thereby 
$$
\Big| \cos( \zeta ) \Re(e^{i\tau} f(y)) -\Re(e^{i\tau} f(y)) -  \sin(  \zeta ) \Im(e^{i\tau} f(y)) \Big|  < \epsilon''|f|
$$ 
for some sufficiently small $\epsilon''$. Note that $f\in L^\infty(D)$, therefore we can derive from \eqref{theorem T alpha coercivity proof eqn1} that
\begin{eqnarray}
&& \left| \langle \Re(e^{i\tau} T_\alpha)h,h\rangle_{L^2(D)} - \langle \Re(e^{i\tau}\widetilde{T}) h,h\rangle_{L^2(D)} \right|  \le \epsilon'''\|h\|^2_{L^2(D)} \label{theorem T alpha coercivity proof eqn2}
\end{eqnarray}
for some sufficiently small $\epsilon'''$. Together with the assumption that $\Re(e^{i\tau}\widetilde{T})$ is coercive with estimate \eqref{lemma T0 coercive proof eqn1}, we have that
$$
\langle \Re(e^{i\tau} T_\alpha) h,h\rangle_{L^2(D)} \ge c\|h\|^2_{L^2(D)},
$$
for some positive constant $c=c'-\epsilon'''>0$. This completes the proof.
\end{proof}

Based on Theorem \ref{theorem T alpha coercivity} and Lemma \ref{lemma T0 coercive}, a sufficient  and less restrictive  condition (that $\Re(e^{i\tau} f(x)) > c_3>0 $ almost everywhere for $x\in D$ with some $\tau$) on $f$ can be found to ensure that $\Re(e^{i\tau} T_\alpha)$ is coercive.  
%%%%%%
\subsection{Generalization of the main results}
Based on the discussion of the multi-frequency operators $\widetilde{F}$ and $F_\alpha$,  we discuss below how to generalize the main Theorems \ref{theorem FM main} and \ref{factorization-based SM theorem}.   In particular, we can prove exactly in the same way  that Theorem \ref{theorem FM main} holds when replacing $F^{1/2}$ by the self-adjoint operator $(F_{\alpha\#})^{1/2}$ or $(\widetilde{F}_\#)^{1/2}$, where $F_{\alpha \#}=\Re(e^{i\tau} F_\alpha)$ and $\widetilde{F}_\#=  \Re(e^{i\tau}\widetilde{F})$. On the other hand, Theorem \ref{factorization-based SM theorem} for the factorization-based sampling method holds when replacing $F$ by  $\Re(e^{i\tau} F_\alpha)$ or $\Re(e^{i\tau}\widetilde{F})$.
}
 
\section*{Acknowledgement}
The author greatly thanks the anonymous referees for their helpful comments that have improved the quality of this manuscript.

%%%%%%%%%%%%%%%%%%%%%%%%%%%%%%%%%%%%%%%%%%
\section{Conclusion} \label{section conclusion}
The factorization method is investigated to treat the inverse source problem in waveguides with a single propagating mode at multiple frequencies. A factorization-based sampling method is also discussed. The main result of the factorization method is to provide both theoretical justifications and efficient imaging algorithms  to image extended sources. One key step is to use the multi-frequency measurements  to introduce the multi-frequency far-field operator, which is carefully designed to incorporate the possibly non-linear dispersion relation. We also emphasize that the measurements under consideration are ``backscattering'' measurements (in contrast to measurements obtained in two opposite directions), and we have shown how to use these less measurements to design appropriate multi-frequency operators to deploy the factorization method. 
% (in contrast to the linear dispersion in the whole space). 
Each multi-frequency operator  can be factorized as $S^*TS$ which plays a fundamental role in the mathematically rigorous justification of the factorization method. A factorization-based sampling method is also discussed, where the factorization $S^*TS$ and a relevant point spread function play important roles. The sampling methods are capable to image the range support of the extended sources. 
 
 Another interesting observation is the potential to image a complete block, since the linear sampling method may not be capable of imaging such a complete block using  far-field measurements at a single frequency. We provide examples to image both a complete source block and a complete obstacle block 
%(which is an inverse obstacle problem as considered in  \cite{monk2012sampling}) 
using the multi-frequency sampling methods. It is shown that the use of one propagating mode at multiple frequencies can efficiently locate a complete block where such efficiency may be of practical interests in long-distance communication optical devices or tunnels.  
%Preliminary numerical examples show the potential to image such complete blocks using the sampling methods.

This work establishes the fundamental result with  a single propagating mode at multiple frequencies. One direction of future work will be to investigate the factorization method in waveguide imaging using multiple  propagating modes to establish theoretical justifications for extended scatterers in  inverse obstacle/medium problems. 
%The extension to multiple mode waveguide is expected since different propagating modes correspond to different observation/incident directions, this is because different propagating modes strike the boundary at different angles (for example, see the explicit expression of the propagating modes in a Neumann boundary waveguide). With those different observation/incident directions, one may expect to image the (convex) support of the scatterers. 
%The inverse obstacle/medium problems in waveguides is expected to work with modified $\psi_z^\epsilon$ and $\psi_z$ (modification in doubling their phases) and the theoretical justifications are yet to be done.

\bibliographystyle{SIAM}

\end{document}